\documentclass[11pt,a4paper]{amsart}
\usepackage{amsmath,amssymb, amsbsy}
\usepackage{color,psfrag}
\usepackage[dvips]{graphicx}
\usepackage{enumerate}
\newcommand{\R}{{\mathbb R}}
\newcommand{\N}{{\mathbb N}}

\newcommand{\e }{\varepsilon}

\newcommand{\HH}{\mathcal{H}}
\newcommand{\eps}{\varepsilon}

\def\ds{\displaystyle}
\renewcommand{\ge }{\geqslant}
\renewcommand{\geq }{\geqslant}
\renewcommand{\le }{\leqslant}
\renewcommand{\leq }{\leqslant}
\def\neweq#1{\begin{equation}\label{#1}}
\def\endeq{\end{equation}}
\def\eq#1{(\ref{#1})}
\newtheorem{theorem}{Theorem}[section]
\newtheorem{proposition}[theorem]{Proposition}
\newtheorem{lemma}[theorem]{Lemma}

\newtheorem{conjecture}[theorem]{Conjecture}

\newtheorem{remark}[theorem]{Remark}
\newtheorem{definition}[theorem]{Definition}

\renewcommand{\arraystretch}{1.5}
\begin{document}

\title[]{Structural instability of nonlinear plates\\
modelling suspension bridges:\\
mathematical answers to some long-standing questions}

\author[Elvise BERCHIO]{Elvise BERCHIO}
\address{\hbox{\parbox{5.7in}{\medskip\noindent{Dipartimento di Scienze Matematiche, \\
Politecnico di Torino,\\
        Corso Duca degli Abruzzi 24, 10129 Torino, Italy. \\[3pt]
        \em{E-mail address: }{\tt elvise.berchio@polito.it}}}}}
\author[Alberto FERRERO]{Alberto FERRERO}
\address{\hbox{\parbox{5.7in}{\medskip\noindent{Dipartimento di Scienze e Innovazione Tecnologica, \\
Universit\`a del Piemonte Orientale ``Amedeo Avogadro'',\\
        Viale Teresa Michel 11, 15121 Alessandria, Italy. \\[3pt]
        \em{E-mail address: }{\tt alberto.ferrero@mfn.unipmn.it}}}}}
\author[Filippo GAZZOLA]{Filippo GAZZOLA}
\address{\hbox{\parbox{5.7in}{\medskip\noindent{Dipartimento di Matematica,\\
Politecnico di Milano,\\
   Piazza Leonardo da Vinci 32, 20133 Milano, Italy. \\[3pt]
        \em{E-mail address: }{\tt filippo.gazzola@polimi.it}}}}}

\date{\today}



\begin{abstract}
We model the roadway of a suspension bridge as a thin rectangular plate and we study in detail its oscillating modes. The plate is assumed to be hinged
on its short edges and free on its long edges. Two different kinds of oscillating modes are found: longitudinal modes and torsional modes. Then we analyze
a fourth order hyperbolic equation describing the dynamics of the bridge. In order to emphasize the structural behavior we consider an isolated equation
with no forcing and damping. Due to the nonlinear behavior of the cables and hangers, a structural instability appears. With a finite dimensional approximation
we prove that the system remains stable at low energies while numerical results show that for larger energies the system becomes unstable.
We analyze the energy thresholds of instability and we show that the model allows to give answers to several questions left open by the Tacoma collapse in 1940.
\end{abstract}

\maketitle

\section{Introduction}

The history of suspension bridges essentially starts a couple of
centuries ago. The first modern suspension bridge is considered to
be the Jacob Creek Bridge, built in Pennsylvania in 1801 and
designed by the Irish judge and engineer James Finley, see
\cite{jfinley} for the patent and the original design. At the same
time, several suspension bridges were erected in the UK, see
(e.g.) the introduction in the seminal book \cite{bleich}. The
political instability due to the French Revolution and to the
Napoleon period kept France slightly delayed. For this reason, M.\
Becquey (Conseiller d'Etat, Directeur G\'en\'eral des Ponts et
Chauss\'ees et des Mines) committed Navier to visit the main
bridges in the UK and to report on their feasibility and
performances. In his detailed report \cite[p.161]{navier}, Navier
wondered about the possible negative effects of the action of the
wind: {\em Les accidens qui r\'esulteraient de cette action ne
peuvent \^etre appr\'eci\'es et pr\'evenus que d'apr\`es les
lumi\`eres fournies par l'observation et l'exp\'erience}.
Unfortunately, he had seen right.\par Many bridges manifested
aerodynamic instability and uncontrolled oscillations leading to
collapses, see e.g.\ \cite{akesson,bridgefailure}. These accidents
are due to many different causes and in this paper we are only
interested about those due to wide unexpected oscillations. We
will give a mathematical explanation for the appearance of
torsional oscillations by analyzing a suitable partial
differential equation modeling the bridge.\par Thanks to the
videos available on the web \cite{tacoma} most people have seen
the spectacular collapse of the Tacoma Narrows Bridge (TNB),
occurred in 1940. In Figure \ref{tacoma12} (picture taken from
\cite[p.6]{ammann})
\begin{figure}[ht]
\begin{center}
{\includegraphics[height=59mm, width=125mm]{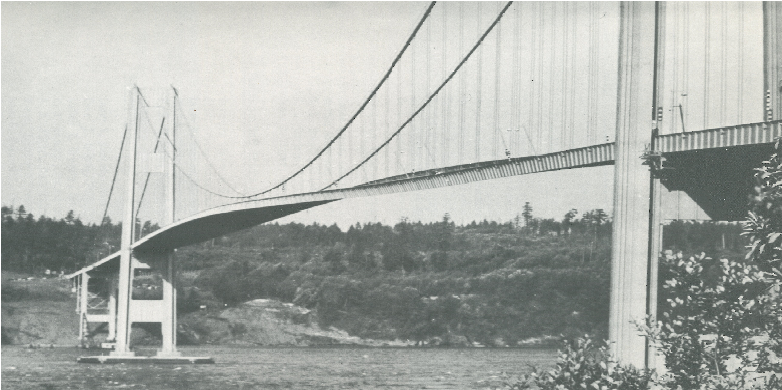}}
\caption{The collapsed Tacoma Narrows Bridge (1940).}\label{tacoma12}
\end{center}
\end{figure}
one can see the roadway of the TNB under a torsional oscillation. This kind of oscillation was considered the main cause of the
collapse \cite{ammann,scott}.
But the appearance of torsional oscillations is not an isolated event occurred only at the TNB. The Brighton Chain Pier was erected in 1823 and
collapsed in 1836: Reid \cite{reid} reported valuable observations and sketched a picture illustrating the collapse see Figure \ref{brighton}
\begin{figure}[ht]
\begin{center}
{\includegraphics[height=99mm, width=125mm]{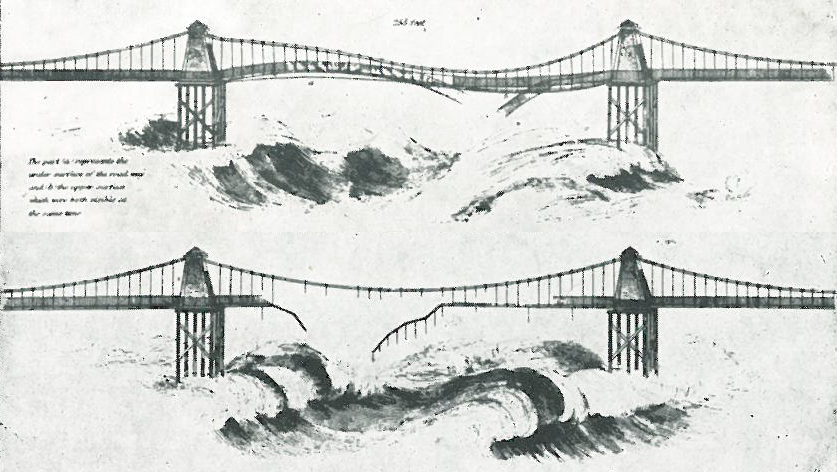}}
\caption{Collapse of the Brighton Chain Pier (1836).}\label{brighton}
\end{center}
\end{figure}
(picture taken from \cite{reid}).
The Wheeling Suspension Bridge was erected in West Virginia in 1849 and collapsed in a violent windstorm in 1854: according to \cite{wheel},
it {\em twisted and writhed, and was dashed almost bottom upward. At last there seemed to be a determined twist along the entire span, about
one half of the flooring being nearly reversed, and down went the immense structure from its dizzy height to the stream below, with an appalling
crash and roar}. Finally, let us mention that Irvine \cite[Example 4.6, p.180]{irvine} describes the collapse of the Matukituki Suspension
Footbridge in New Zealand (occurred in 1977, just twelve days after completion) by writing that {\em the deck persisted in lurching and twisting
wildly until failure occurred, and for part of the time a node was noticeable at midspan}. This description is completely similar to what can be
seen in Figures \ref{tacoma12} and \ref{brighton} as well as in the video \cite{tacoma}. These are just few examples aiming to show that the very
same instability was observed in several different bridges.\par
These accidents raised some fundamental questions of deep interest also for mathematicians. Longitudinal oscillations are to be expected in
suspension bridges but
\begin{center}
{\bf (Q1) why do longitudinal oscillations suddenly transform into torsional oscillations?}
\end{center}
This question has drawn the attention of both mathematicians and engineers but, so far, no unanimously accepted response has been found. The distinguished
civil and aeronautical engineer Robert Scanlan \cite[p.209]{scanlan2} attributes the appearance of torsional oscillations to {\em some fortuitous condition}.
The word ``fortuitous'' highlights a lack of rigorous explanations and, according to \cite{scott}, no real progress has been done in subsequent years.\par
The above collapses also show that the torsional oscillation has a particular shape, with a node at midspan. And it seems that this particular kind of
torsional oscillation is the only one ever seen in suspension bridges. From the Official Report \cite[p.31]{ammann} we quote {\em Prior to 10:00 A.M.\ on
the day of the failure, there were no recorded instances of the oscillations being otherwise than the two cables in phase and with no torsional motions} whereas
from Smith-Vincent \cite[p.21]{tac2} we quote {\em the only torsional mode which developed under wind action on the bridge or on the model is that with a
single node at the center of the main span}. This raises a further natural question:
\begin{center}
{\bf (Q2) why do torsional oscillations appear with a node at midspan?}
\end{center}

According to Eldridge \cite[V-3]{ammann}, a witness on the day of the TNB collapse, {\em the bridge appeared to be behaving in the customary
manner} and the motions {\em were considerably less than had occurred many times before}. From \cite[p.20]{ammann} we also learn that in the months
prior to the collapse {\em one principal mode of oscillation prevailed} and that {\em the modes of oscillation frequently changed}.
In particular, Farquharson \cite[V-10]{ammann} witnessed the collapse and wrote that {\em the motions, which a moment before had involved a
number of waves (nine or ten) had shifted almost instantly to two}. This raises a third natural question:
\begin{center}
{\bf (Q3) are there longitudinal oscillations which are more prone
to generate torsional oscillations?}
\end{center}

The purpose of this paper is to use the semilinear plate model
developed in \cite{fergaz} and to adapt it to a suspension bridge
having the same parameters as the collapsed TNB. By analyzing both
theoretically and numerically this model, we will give an answer
to the above questions {\bf (Q1)}, {\bf (Q2)} and {\bf (Q3)}.\par
This paper is organized as follows. In Section \ref{nonmod} we
recall and slightly modify the model introduced in \cite{fergaz},
in particular we discuss the nonlinear restoring force due to the
hangers+cables system. In Section \ref{modes} we study in great
detail the oscillating modes of the plate, according to the TNB
parameters. In Section \ref{s:4} we analyze the full evolution
equation in the case where the system is isolated: we obtain a
fourth order hyperbolic equations and we show that the
initial-boundary-value problem is well posed. Then we define what
we mean by torsional stability and we state two sufficient
conditions for the stability. In Section \ref{numres} we
numerically compute the thresholds of stability, according to our
definition. In Section \ref{full} we validate our results from
several points of view: we show that the linearization and the
uncoupling procedures do not alter the results and that a full
numerical analysis does not give significantly different
responses. Sections \ref{stable proof} and \ref{secondproof} are
devoted to the proofs of the stability results. Finally, in
Section \ref{conclusions} we afford an answer to the above
questions.

\section{A nonlinear model for a dynamic suspension bridge}\label{nonmod}

We follow the mathematical model suggested in \cite{fergaz} by modifying it in some aspects. We view the roadway (or deck) of a suspension bridge as a
long narrow rectangular thin plate hinged at the two opposite short edges and free on the remaining two edges. Let $L$ denote its length and
$2\ell$ denote its width; a realistic assumption is that $2\ell\cong\frac{L}{100}$. The rectangular plate $\Omega\subset\R^2$ is then
$$\Omega=(0,L)\times(-\ell,\ell)\, .$$
Let us first discuss different positions of the plate depending on the forces acting on it. If the plate had no mass (as a sheet of paper) and there were no
loads acting on the plate, it would take the horizontal equilibrium position $u_0$, see Figure \ref{three}. If the plate was only subject to its own weight
$w$ (dead load) it would take a $\cup$-position such as $u_w$ in Figure \ref{three}. If the plate had no weight but it was subject to the restoring
force of the cables-hangers system, it would take a $\cap$-position such as $u_h$: this is also the position of the lower endpoints of the hangers before
the roadway is installed. If both the weight and the action of the hangers are considered, the two effects cancel and the equilibrium position $u_0\equiv0$
is recovered.
\begin{figure}[th]
\begin{center}
{\includegraphics[height=23mm, width=126mm]{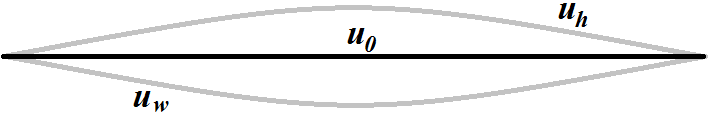}}
\caption{Different positions of the bridge.}\label{three}
\end{center}
\end{figure}

Since the bending energy of the plate vanishes when it is in position $u_0\equiv0$, the unknown function should be the displacement of the plate with respect
to the equilibrium $u_0$. Augusti-Sepe \cite{sepe1} (see also \cite{sepe2}) view the restoring force at the endpoints of a cross-section of the roadway as
composed by two connected springs, the top one representing the action of the sustaining cable and the bottom one (connected with the roadway)
representing the hangers, see Figure \ref{cableshangers}.
\begin{figure}[ht]
\begin{center}
{\includegraphics[height=30mm, width=42mm]{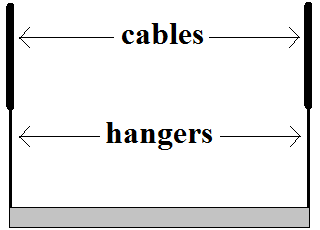}}
\caption{The cables+hangers system modeled with two connected springs.}\label{cableshangers}
\end{center}
\end{figure}

The action of the cables is considered by Bartoli-Spinelli \cite[p.180]{bartoli} the main cause of the nonlinearity of the restoring force: they suggest
quadratic and cubic perturbations of a linear behavior. If $u$ denotes the downwards displacement of the deck, here we simply take
\neweq{gu}
g(u)=k_1\, u+k_2\, u^3
\endeq
for some $k_1,k_2>0$ depending on the elasticity of the cables and hangers. Let us mention that Plaut-Davis \cite[$\S$ 3.5]{plautdavis} make
the same choice.\par
The action of the hangers on the roadway is confined in the union of two thin strips parallel and adjacent to the two long edges of the plate $\Omega$,
that is, in a set of the type
\neweq{omega}
\omega:=(0,L)\times [(-\ell,-\ell+\eps)\cup (\ell-\eps,\ell)]
\endeq
with $\eps>0$ small compared to $\ell$. Summarizing, we take as restoring force and potential due to the cables-hangers system
\neweq{particular1}
h(y,u)=\Upsilon(y)\, \Big(k_1\, u+k_2\, u^3\Big)\, ,\qquad
H(y,u)=\int_0^u h(y,\tau)d\tau=\Upsilon(y)\, \left(\frac{k_1}{2}\, u^2+\frac{k_2}{4}\, u^4\right)\, ,
\endeq
where $\Upsilon$ is the characteristic function of $(-\ell,-\ell+\eps)\cup (\ell-\eps,\ell)$.\par
The derivation of the bending energy of an elastic plate goes back to Kirchhoff \cite{Kirchhoff} and Love \cite{Love}, see also \cite{fergaz} for a
synthesized form. The total static energy of the bridge is obtained by adding the potential energy from \eq{particular1} to the bending energy of the plate:
\neweq{static}
\mathbb E_T(u)=\frac{E\, d^3}{12(1-\sigma^2)}\int_{\Omega}\left(\frac12 \left(\Delta u\right)^2+(\sigma-1)\det(D^2u)\right)+\int_\Omega\big(H(y,u)-fu\big)
\endeq
where $f$ is an external force, $d$ denotes the thickness of the plate, $E$ is the Young modulus and $\sigma$ is the Poisson ratio.
For a plate, the Poisson ratio is the negative ratio of transverse to axial strain: when a material is compressed in one direction, it tends to
expand in the other two directions. The Poisson ratio $\sigma$ is a measure of this effect, it is the fraction of expansion divided by the fraction
of compression for small values of these changes. Usually one has
\neweq{sigma}
0<\sigma<\frac{1}{2}\, .
\endeq

A functional space where the energy $\mathbb E_T$ is well-defined is
$$H^2_*(\Omega):=\Big\{w\in H^2(\Omega);\, w=0\mbox{ on }\{0,L\}\times (-\ell,\ell)\Big\}$$
which is a Hilbert space when endowed with the scalar product
\neweq{scalarp}
(u,v)_{H^2_*}:=\int_\Omega \left[\Delta u\Delta v+(1-\sigma)(2u_{xy}v_{xy}-u_{xx}v_{yy}-u_{yy}v_{xx})\right]\, dxdy\, ,
\endeq
see \cite[Lemma 4.1]{fergaz}. We also consider
$$\HH(\Omega):=\mbox{ the dual space of }H^2_*(\Omega)$$
and we denote by $\langle\cdot,\cdot\rangle$ the corresponding duality. Since we are in the plane, $H^2(\Omega)\subset C^0(\overline{\Omega})$
so that the condition on $\{0,L\}\times(-\ell,\ell)$ introduced in the definition of $H^2_*(\Omega)$ is satisfied pointwise.
If $f\in L^1(\Omega)$ then the functional $\mathbb E_T$ is well-defined in $H^2_*(\Omega)$, while if $f\in\HH(\Omega)$ we need to replace
$\int_\Omega fu$ with $\langle f,u\rangle$ although we will not mention this in the sequel.\par
If the load also depends on time, $f=f(x,y,t)$, and if $m$ denotes the mass density of the plate, then the
kinetic energy of the plate should be added to the static energy \eq{static}:
\neweq{totalenergy}
{\mathcal{E}_u}(t):=\frac{m}{2}\int_\Omega  u_t^2+\frac{E\, d^3}{12(1-\sigma^2)}
\int_\Omega\left(\frac{(\Delta u)^2}{2}+(\sigma-1)\det(D^2u)\right)+\int_\Omega\big(H(y,u)-fu\big)\, .
\endeq
This is the total energy of a nonlinear dynamic bridge. As for the action, one has to take the difference between kinetic
energy and potential energy and integrate over an interval of time $[0,T]$:
\begin{equation*}
\mathcal A(u):=\int_0^T\left[\frac{m}{2}\int_\Omega u_t^2 -\frac{E\, d^3}{12(1-\sigma^2)}
\int_\Omega\left(\frac{(\Delta u)^2}{2}+(\sigma-1)\det(D^2u)\right)-\int_\Omega\big(H(y,u)-fu\big)\right]\, dt\, .
\end{equation*}
The equation of the motion of the bridge is obtained by taking the critical points of the functional $\mathcal A$:
$$m\, u_{tt}+\frac{E\, d^3}{12(1-\sigma^2)}\, \Delta^2 u+h(y,u)=f \qquad \text{in } \Omega\times(0,T)\, .$$
Due to internal friction, we add a damping term and obtain
\neweq{wave-2}
\left\{\begin{array}{ll}
\!\! mu_{tt}\!+\!\delta u_t\!+\!\frac{E\, d^3}{12(1-\sigma^2)}\Delta^2 u\!+\!h(y,u)\!=\!f & \text{in } \Omega\!\times\!(0,T) \\
\!\! u(0,y,t)\!=\!u_{xx}(0,y,t)\!=\!u(L,y,t)\!=\!u_{xx}(L,y,t)\!=\!0 & \text{for } (y,t)\!\in\!(-\ell,\ell)\!\times\!(0,T) \\
\!\! u_{yy}(x,\pm\ell,t)\!+\!\sigma u_{xx}(x,\pm\ell,t)\!=\!0 & \text{for }(x,t)\!\in\!(0,L)\!\times\!(0,T)\\
\!\!u_{yyy}(x,\pm\ell,t)\!+\!(2-\sigma)u_{xxy}(x,\pm\ell,t)\!=\!0 & \text{for }(x,t)\!\in\!(0,L)\!\times\!(0,T)\\
\!\! u(x,y,0)\!=\!u_0(x,y)\, ,\quad u_t(x,y,0)\!=\!u_1(x,y)& \text{for }(x,y)\!\in\! \Omega
\end{array}\right.\endeq
where $\delta$ is a positive constant. We refer to \cite{fergaz} for the derivation of the boundary conditions.\par
For a different model of suspension bridges, similar to the one considered in \cite{arga}, Irvine \cite[p.176]{irvine} ignores damping of both structural
and aerodynamic origin. His purpose is to simplify as much as possible the model by maintaining its essence, that is, the conceptual design of bridges.
Here we follow this suggestion and consider the {\em isolated version} of \eq{wave-2} for which global existence is expected ($T=\infty$).
This isolated version of \eq{wave-2} reads
\begin{equation}\label{tac}
m w_{tt}\!+\! \frac{E\, d^3}{12(1-\sigma^2)}\Delta^2 w\!+\!\Upsilon(y)(k_1w+k_2w^3)\!=\!0\qquad\text{for }(x,y)\in(0,L)\times\left(-\ell,\ell\right),\,t>0\,,
\end{equation}
where $w$ denotes the downwards vertical displacement and all the constants are defined in Section \ref{nonmod}.\par
In order to set up a reliable model, we consider the lengths of the plate as in the collapsed TNB. According to \cite[p.11]{ammann}, we have
\neweq{proportion}
L=2800\, \mbox{ft.}\approx 853.44\, m\, ,\quad2\ell=39\, \mbox{ft.}\approx 11.89\, m\, ,\qquad\mbox{that is,}\qquad
\frac{2\ell}{L}=\frac{39}{2800}\approx\frac{1}{75}=\frac{\, \frac{2\pi}{150}\, }{\pi}\, .
\endeq
Therefore, we may scale the plate $(0,L)\times(-\ell,\ell)$ to $(0,\pi)\times(-\tfrac{\pi}{150},\tfrac{\pi}{150})$. Then we take the amplitude of the
strip $\omega$ (see \eq{omega}) containing the hangers also as at the TNB, see \cite[p.11]{ammann}:
\neweq{epsilon}
\eps=\frac{\pi}{1500}\, .
\endeq
Referring to Section \ref{numres} for the values of the involved parameters, we put
\neweq{gamma}
w(x,y,t)=\sqrt{\frac{k_1}{k_2}}\ u\left(\frac{\pi x}{L},\frac{\pi y}{L},\sqrt{\frac{k_1}{m}}\, t\right)\qquad\mbox{and}\qquad
\gamma= \frac{E\, d^3}{12k_1(1-\sigma^2)}\, \frac{\pi^4}{L^4}\, .
\endeq
Then \eq{tac} becomes an equation where the only parameter is the coefficient of the biharmonic term:
$$
u_{tt}+\gamma\, \Delta^2 u+\Upsilon(y)(u+u^3)=0\qquad\text{in }(0,\pi)\times\left(-\frac{\pi}{150},\frac{\pi}{150}\right)\times\R_+
$$
and $\Upsilon$ is the characteristic function of the set $(-\frac{\pi}{150},-\frac{3\pi}{500})\cup(\frac{3\pi}{500},\frac{\pi}{150})$.
Finally, we notice that for metals the value of $\sigma$ lies around $0.3$, see \cite[p.105]{Love}, while for concrete we have $0.1<\sigma<0.2$.
Since the suspended structure of the TNB consisted of a ``mixture'' of concrete and metal (see \cite[p.13]{ammann}), we take
\neweq{verosigma}
\sigma=0.2\, .
\endeq

Then by using \eq{gamma} we find the dimensionless version of the problem under study
\neweq{wave-3}
\left\{\begin{array}{ll}
\!\! u_{tt}\!+\!\gamma\Delta^2 u\!+\!\Upsilon(y)(u+u^3)\!=\!0 & \text{in } \Omega\!\times\!(0,\infty) \\
\!\! u(0,y,t)\!=\!u_{xx}(0,y,t)\!=\!u(\pi,y,t)\!=\!u_{xx}(\pi,y,t)\!=\!0 & \text{for } (y,t)\!\in\!(-\frac{\pi}{150},\frac{\pi}{150})
\!\times\!(0,\infty) \\
\!\!u_{yy}(x,\pm\frac{\pi}{150},t)\!+\!0.2\cdot u_{xx}(x,\pm\frac{\pi}{150},t)\!=\!0 & \text{for }(x,t)\!\in\!(0,\pi)\!\times\!(0,\infty)\\
\!\!u_{yyy}(x,\pm\frac{\pi}{150},t)\!+\!1.8\cdot u_{xxy}(x,\pm\frac{\pi}{150},t)\!=\!0 & \text{for }(x,t)\!\in\!(0,\pi)\!\times\!(0,\infty)\\
\!\! u(x,y,0)\!=\!u_0(x,y)\, ,\quad u_t(x,y,0)\!=\!u_1(x,y)& \text{for }(x,y)\!\in\! \Omega
\end{array}\right.\endeq
where $\Omega=(0,\pi)\times(-\frac{\pi}{150},\frac{\pi}{150})$. The initial-boundary value problem \eq{wave-3} is isolated, which means that it
has a conserved quantity. This quantity is the energy introduced in \eq{totalenergy} which is constant in time:
\neweq{energy}
E(u)=\int_\Omega \frac12 u_t^2\, dxdy+\int_\Omega\left(\frac{\gamma}{2} (\Delta u)^2+\frac{4\gamma}{5} (u_{xy}^2-u_{xx}u_{yy})+
\Upsilon(y)\,\left(\frac{u^2}{2}+\frac{u^4}{4}\right)\right)\, dxdy\, .
\endeq

\section{The eigenfunctions of the linearized problem}\label{modes}

For the rectangular plate $\Omega=(0,\pi)\times(-\ell,\ell)$ with
Poisson ratio $\sigma$ we are interested in the eigenfunctions of
the eigenvalue problem
\begin{equation}\label{eq:Eigenvalue}
\begin{cases}
\Delta^2 w=\lambda w & \qquad \text{in } \Omega \\
w(0,y)=w_{xx}(0,y)=w(\pi,y)=w_{xx}(\pi,y)=0 & \qquad \text{for } y\in (-\ell,\ell) \\
w_{yy}(x,\pm\ell)+\sigma
w_{xx}(x,\pm\ell)=w_{yyy}(x,\pm\ell)+(2-\sigma)w_{xxy}(x,\pm\ell)=0
& \qquad \text{for } x\in (0,\pi)\, .
\end{cases}
\end{equation}
Problem \eqref{eq:Eigenvalue} admits the following variational formulation: a nontrivial function $w\in H^2_*(\Omega)$ is an eigenfunction
of \eqref{eq:Eigenvalue} if there exists $\lambda\in\R$ (an eigenvalue) such that
$$
\int_\Omega \left[\Delta w \Delta v+(1-\sigma)(2w_{xy}v_{xy}-w_{xx}v_{yy}-w_{yy}v_{xx})-\lambda
wv\right]\, dxdy=0\qquad \text{for all }v\in H^2_*(\Omega)\, .
$$

We recall from \cite[Theorem 7.6]{fergaz} a statement describing the whole spectrum and characterizing the eigenfunctions.
It is shown there that the eigenfunctions may have one of the following forms:

\begin{proposition}\label{eigenvalue}
Assume \eqref{sigma}. Then the set of eigenvalues of \eqref{eq:Eigenvalue} may be ordered in an increasing sequence
$\{\lambda_k\}$ of strictly positive numbers diverging to $+\infty$ and any eigenfunction belongs to
$C^\infty(\overline\Omega)$; the set of eigenfunctions of \eqref{eq:Eigenvalue} is a complete system in $H^2_*(\Omega)$. Moreover:\par\noindent
$(i)$ for any $m\ge1$, there exists a unique eigenvalue $\lambda=\mu_{m,1}\in((1-\sigma)^2m^4,m^4)$ with corresponding eigenfunction
$$\left[\big[\mu_{m,1}^{1/2}-(1-\sigma)m^2\big]\, \tfrac{\cosh\Big(y\sqrt{m^2+\mu_{m,1}^{1/2}}\Big)}{\cosh\Big(\ell\sqrt{m^2+\mu_{m,1}^{1/2}}\Big)}+
\big[\mu_{m,1}^{1/2}+(1-\sigma)m^2\big]\, \tfrac{\cosh\Big(y\sqrt{m^2-\mu_{m,1}^{1/2}}\Big)}{\cosh\Big(\ell\sqrt{m^2-\mu_{m,1}^{1/2}}\Big)}\right]\sin(mx)\, ;$$
$(ii)$ for any $m\ge1$, there exist infinitely many eigenvalues $\lambda=\mu_{m,k}>m^4$ ($k\ge2$) with corresponding eigenfunctions
$$
\left[\big[\mu_{m,k}^{1/2}-(1-\sigma)m^2\big]\, \tfrac{\cosh\Big(y\sqrt{\mu_{m,k}^{1/2}+m^2}\Big)}{\cosh\Big(\ell\sqrt{\mu_{m,k}^{1/2}+m^2}\Big)}
+\big[\mu_{m,k}^{1/2}+(1-\sigma)m^2\big]\, \tfrac{\cos\Big(y\sqrt{\mu_{m,k}^{1/2}-m^2}\Big)}{\cos\Big(\ell\sqrt{\mu_{m,k}^{1/2}-m^2}\Big)}\right]\sin(mx)\, ;
$$
$(iii)$ for any $m\ge1$, there exist infinitely many eigenvalues $\lambda=\nu_{m,k}>m^4$ ($k\ge2$) with corresponding eigenfunctions
$$
\left[\big[\nu_{m,k}^{1/2}-(1-\sigma)m^2\big]\, \tfrac{\sinh\Big(y\sqrt{\nu_{m,k}^{1/2}+m^2}\Big)}{\sinh\Big(\ell\sqrt{\nu_{m,k}^{1/2}+m^2}\Big)}
+\big[\nu_{m,k}^{1/2}+(1-\sigma)m^2\big]\, \tfrac{\sin\Big(y\sqrt{\nu_{m,k}^{1/2}-m^2}\Big)}{\sin\Big(\ell\sqrt{\nu_{m,k}^{1/2}-m^2}\Big)}\right]\sin(mx)\, ;
$$
$(iv)$ for any $m\ge1$ satisfying $\ell m\sqrt 2\, \coth(\ell m\sqrt2 )>\left(\frac{2-\sigma}{\sigma}\right)^2$ there exists an
eigenvalue $\lambda=\nu_{m,1}\in(\mu_{m,1},m^4)$ with corresponding eigenfunction
$$\left[\big[\nu_{m,1}^{1/2}-(1-\sigma)m^2\big]\, \tfrac{\sinh\Big(y\sqrt{m^2+\nu_{m,1}^{1/2}}\Big)}{\sinh\Big(\ell\sqrt{m^2+\nu_{m,1}^{1/2}}\Big)}
+\big[\nu_{m,1}^{1/2}+(1-\sigma)m^2\big]\, \tfrac{\sinh\Big(y\sqrt{m^2-\nu_{m,1}^{1/2}}\Big)}{\sinh\Big(\ell\sqrt{m^2-\nu_{m,1}^{1/2}}\Big)}\right]
\sin(mx)\, .$$
Finally, if the unique positive solution $s>0$ of the equation
\neweq{iff}
\tanh(\sqrt{2}s\ell)=\left(\frac{\sigma}{2-\sigma}\right)^2\, \sqrt{2}s\ell
\endeq
is not an integer, then the only eigenvalues and eigenfunctions are the ones given in $(i)-(iv)$.
\end{proposition}

Of course, \eq{iff} has probability 0 to occur in a real bridge; if it occurs, there is an additional eigenvalue and eigenfunction, see \cite{fergaz}.
The eigenvalues $\lambda_k$ are solutions of explicit equations. More precisely:\par\noindent
$(i)$ the eigenvalue $\lambda=\mu_{m,1}$ is the unique value $\lambda\in((1-\sigma)^2m^4,m^4)$ such that
$$
\sqrt{m^2\!-\!\lambda^{1/2}}\big(\lambda^{1/2}\!+\!(1\!-\!\sigma)m^2\big)^2\tanh(\ell\sqrt{m^2\!-\!\lambda^{1/2}})\!=\!\sqrt{m^2\!+\!\lambda^{1/2}}
\big(\lambda^{1/2}\!-\!(1\!-\!\sigma)m^2\big)^2\tanh(\ell\sqrt{m^2\!+\!\lambda^{1/2}})\, ;
$$
$(ii)$ the eigenvalues $\lambda=\mu_{m,k}$ ($k\ge2$) are the solutions $\lambda>m^4$ of the equation
$$
\sqrt{\lambda^{1/2}\!-\!m^2}\big(\lambda^{1/2}\!+\!(1\!-\!\sigma)m^2\big)^2\tan(\ell\sqrt{\lambda^{1/2}\!-\!m^2})\!=\!-\!\sqrt{\lambda^{1/2}\!+\!m^2}
\big(\lambda^{1/2}\!-\!(1\!-\!\sigma)m^2\big)^2\tanh(\ell\sqrt{\lambda^{1/2}\!+\!m^2})\, ;
$$
$(iii)$ the eigenvalues $\lambda=\nu_{m,k}$ ($k\ge2$) are the solutions $\lambda>m^4$ of the equation
$$
\sqrt{\lambda^{1/2}\!-\!m^2}\big(\lambda^{1/2}\!+\!(1\!-\!\sigma)m^2\big)^2\tanh(\ell\sqrt{\lambda^{1/2}\!+\!m^2})\!=\!\sqrt{\lambda^{1/2}\!+\!m^2}
\big(\lambda^{1/2}\!-\!(1\!-\!\sigma)m^2\big)^2\tan(\ell\sqrt{\lambda^{1/2}\!-\!m^2})\, ;
$$
$(iv)$ the eigenvalue $\lambda=\nu_{m,1}$ is the unique value $\lambda\in((1-\sigma)^2m^4,m^4)$ such that
$$
\sqrt{m^2\!-\!\lambda^{1/2}}\big(\lambda^{1/2}\!+\!(1\!-\!\sigma)m^2\big)^2\tanh(\ell\sqrt{\lambda^{1/2}\!+\!m^2})\!=\!\sqrt{\lambda^{1/2}\!+\!m^2}
\big(\lambda^{1/2}\!-\!(1\!-\!\sigma)m^2\big)^2\tanh(\ell\sqrt{\lambda^{1/2}\!-\!m^2})\, .
$$

The least eigenvalue is $\lambda_1=\mu_{1,1}$: the corresponding
eigenfunction is of one sign over $\Omega$ and this fact is by far
nontrivial. It is well-known that the first eigenfunction of some
biharmonic problems may change sign. When $\Omega$ is a square,
Coffman \cite{coffman} proved that the first eigenfunction of the
clamped plate problem changes sign, see also \cite{kkm} for more
general results. Moreover, Knightly-Sather \cite[Section
3]{knight} show that the buckling eigenvalue problem for a fully
hinged (simply supported) rectangular plate, that is with
$u=\Delta u=0$ on the four edges, may admit a least eigenvalue of
multiplicity 2. Hence, the positivity of the first eigenfunction
of \eq{eq:Eigenvalue} is not for free. Due to the
$L^2$-orthogonality of eigenfunctions, it is the only positive
eigenfunction of \eq{eq:Eigenvalue}.\par The eigenfunctions in
$(i)-(ii)$ are even with respect to $y$ whereas the eigenfunctions
in $(iii)-(iv)$ are odd. We call {\bf longitudinal eigenfunctions}
the eigenfunctions of the kind $(i)-(ii)$ and {\bf torsional
eigenfunctions} the eigenfunctions of the kind $(iii)-(iv)$. Since
$\ell$ is small, the former are essentially of the kind
$c_m\sin(mx)$ whereas the latter are of the kind $c_my\sin(mx)$.
The pictures in Figure \ref{kable} display the first two
longitudinal eigenfunctions (approximately described by
$c_1\sin(x)$ and $c_2\sin(2x)$) and the second torsional
eigenfunction (approximately described by $c_2y\sin(2x)$). In each
picture the displacement of the roadway is compared with
equilibrium.

\begin{figure}[ht]
\begin{center}
{\includegraphics[height=29mm, width=76mm]{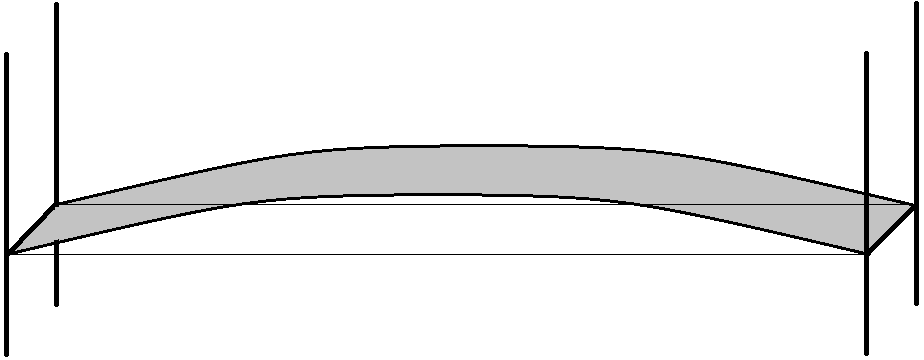}}\qquad{\includegraphics[height=29mm, width=76mm]{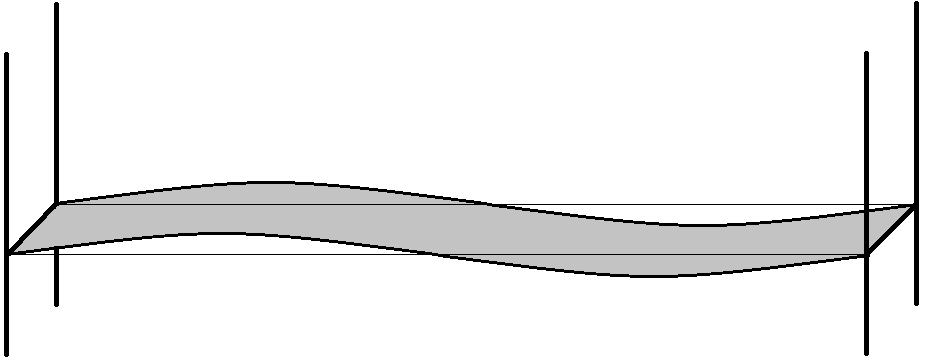}}

{\includegraphics[height=29mm, width=76mm]{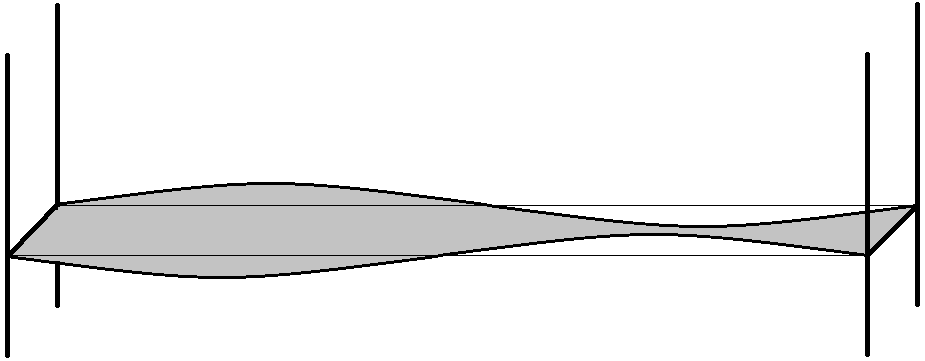}}

\caption{Some possible oscillations of a bridge roadway.}\label{kable}
\end{center}
\end{figure}

Of particular interest is the lower picture in Figure \ref{kable}
which corresponds to a torsional eigenfunction with a node at
midspan. This is precisely the behavior of the TNB prior to its
collapse in November 1940, see the video \cite{tacoma} and Figure
\ref{tacoma12}. This is also the behavior of the Brighton Chain
Pier, see Figure \ref{brighton}, and of the other bridges
described in the introduction. With the notations of Proposition
\ref{eigenvalue}, this common behavior of suspension bridges may
be rephrased as follows
\renewcommand{\arraystretch}{1.1}
\neweq{rephrase}
\begin{array}{ccc}
\mbox{the oscillations causing the collapse of a suspension bridge}\\
\mbox{are of the kind $c_2y\sin(2x)$, as represented in the bottom picture of Figure \ref{kable},}\\
\mbox{and correspond to the eigenvalue $\nu_{2,2}$, as given by Proposition \ref{eigenvalue}.}
\end{array}
\endeq\renewcommand{\arraystretch}{1.5}

We remark that the eigenfunction corresponding to the eigenvalue $\nu_{1,1}$ is of the kind $c_2y\sin(x)$ but this eigenvalue in general does not exist
since the inequality in Proposition \ref{eigenvalue} $(iv)$ is usually satisfied only for large $m$.\par
By \eq{proportion} we may take
\neweq{ellpi}
\ell=\frac{\pi}{150}\, .
\endeq
With the choices in \eq{verosigma} and \eq{ellpi} we numerically
obtained the eigenvalues of \eq{eq:Eigenvalue} as reported in
Table \ref{tableigen}. We only quote the least 16 eigenvalues
because we are mainly interested in the second torsional
eigenvalue which is, precisely, the 16th.

\begin{table}[ht]
\begin{center}
\begin{tabular}{|c|c|c|c|c|c|c|c|c|}
\hline
eigenvalue & $\lambda_1$ & $\lambda_2$ & $\lambda_3$ & $\lambda_4$ & $\lambda_5$ & $\lambda_6$ & $\lambda_7$ & $\lambda_8$\\
\hline
kind & $\mu_{1,1}$ & $\mu_{2,1}$ & $\mu_{3,1}$ & $\mu_{4,1}$ & $\mu_{5,1}$ & $\mu_{6,1}$  & $\mu_{7,1}$ & $\mu_{8,1}$\\
\hline
$\sqrt{\rm{eigenvalue}}\approx$ & $0.98$ & $3.92$ & $8.82$ & $15.68$ & $24.5$ & $35.28$ & $48.02$ & $62.73$\\
\hline
\end{tabular}
\par\smallskip
\begin{tabular}{|c|c|c|c|c|c|c|c|c|}
\hline
eigenvalue & $\lambda_9$ & $\lambda_{10}$ & $\lambda_{11}$ & $\lambda_{12}$ & $\lambda_{13}$ & $\lambda_{14}$ & $\lambda_{15}$ & $\lambda_{16}$\\
\hline
kind & $\mu_{9,1}$  & $\mu_{10,1}$ & $\nu_{1,2}$ & $\mu_{11,1}$ & $\mu_{12,1}$ & $\mu_{13,1}$ & $\mu_{14,1}$ & $\nu_{2,2}$\\
\hline
$\sqrt{\rm{eigenvalue}}\approx$ & $79.39$ & $98.03$ & $104.61$ & $118.62$ & $141.19$ & $165.72$ & $192.21$ & $209.25$\\
\hline
\end{tabular}
\caption{Approximate value of the least 16 eigenvalues of \eq{eq:Eigenvalue} for $\sigma=0.2$ and $\ell=\frac{\pi}{150}$.}\label{tableigen}
\end{center}
\end{table}

Our results are obtained with the parameters of the TNB, see
\eq{proportion}, \eq{verosigma} and \eq{ellpi}. As already
mentioned in the introduction, Farquharson \cite[V-10]{ammann}
witnessed the collapse and wrote that {\em the motions, which a
moment before had involved a number of waves (nine or ten) had
shifted almost instantly to two}. Note that the longitudinal
eigenvalue immediately preceding the least torsional eigenvalue is
$\mu_{10,1}$: it involves the function $\sin(10x)$ which has
precisely ``ten waves''. This explains why at the TNB the
torsional instability occurred when the bridge was longitudinally
oscillating like $\sin(10x)$. Then, if no constraint acts on the
deck, the energy should transfer to the eigenfunction
corresponding to $\nu_{1,2}$ which has a behavior like $y\sin(x)$.
Among longitudinal eigenfunctions, the tenth is the most prone to
torsional instability since the ratio between torsional
eigenvalues and longitudinal eigenvalues is minimal (close to 1)
precisely for the tenth longitudinal eigenfunction: why small
ratios yield strong instability is explained for a simplified
model through Mathieu equations in \cite{bgz}. However, in the
case of a bridge, the sustaining cable yields a serious
constraint. With a rude approximation, the cables may be
considered as inextensible. A better point of view is that they
are only ``weakly extensible'', which means that their elongation
cannot bee too large. In Figure \ref{duepponti}
\begin{figure}[ht]
\begin{center}
 {\includegraphics[height=25mm, width=160mm]{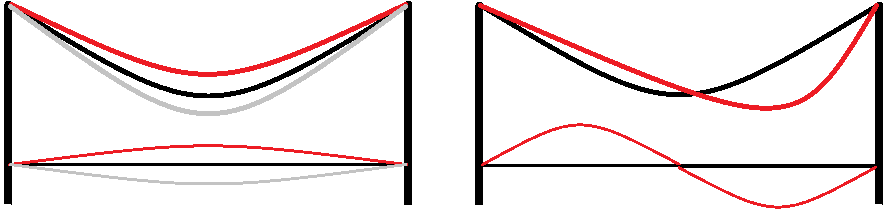}}
\caption{Elongation of the cable generated by the oscillations of the deck.}\label{duepponti}
\end{center}
\end{figure}
we represent the deformation of a cable in the two situations
where the deck behaves like $\sin(x)$ and like $\sin(2x)$. It
turns out that the no-noded behavior $\sin(x)$ (on the left) only
allows small vertical displacements of the deck (between the grey
and red positions) and, therefore, small torsional oscillations of
the kind $y\sin(x)$. This is confirmed by numerical experiments
for models where the cable plays a dominant role, see
\cite{bergotcivati} where it is shown that, basically, ``the first
mode does not exist'' in actual bridges. On the contrary, the
one-noded behavior $\sin(2x)$ allows much larger torsional
oscillations of the kind $y\sin(2x)$, see the right picture. Our
explanation of the transition described by Farquharson is that
when the oscillation $\sin(10x)$ became sufficiently large,
reaching the threshold of the torsional instability, the cable
forced the transition to the eigenfunction $y\sin(2x)$ instead of
$y\sin(x)$. This gives a sound explanation to \eq{rephrase} and a
first answer to {\bf (Q2)}, see Section \ref{numres}. \par If we
slightly modify the choices in \eq{verosigma} and \eq{ellpi},
still in the range of the TNB, we obtain the eigenvalues of
\eq{eq:Eigenvalue} as reported in Table \ref{tableigen2}.

\begin{table}[ht]
\begin{center}
\begin{tabular}{|c|c|c|c|c|c|c|c|c|}
\hline
eigenvalue & $\lambda_1$ & $\lambda_2$ & $\lambda_3$ & $\lambda_4$ & $\lambda_5$ & $\lambda_6$ & $\lambda_7$ & $\lambda_8$\\
\hline
kind & $\mu_{1,1}$ & $\mu_{2,1}$ & $\mu_{3,1}$ & $\mu_{4,1}$ & $\mu_{5,1}$ & $\mu_{6,1}$  & $\mu_{7,1}$ & $\mu_{8,1}$\\
\hline
$\sqrt{\rm{eigenvalue}}\approx$ & $0.97$ & $3.87$ & $8.71$ & $15.49$ & $24.21$ & $34.87$ & $47.46$ & $62$\\
\hline
\end{tabular}
\par\smallskip
\begin{tabular}{|c|c|c|c|c|c|c|c|c|}
\hline
eigenvalue & $\lambda_9$ & $\lambda_{10}$ & $\lambda_{11}$ & $\lambda_{12}$ & $\lambda_{13}$ & $\lambda_{14}$ & $\lambda_{15}$ & $\lambda_{16}$\\
\hline
kind & $\mu_{9,1}$  & $\mu_{10,1}$ & $\nu_{1,2}$ & $\mu_{11,1}$ & $\mu_{12,1}$ & $\mu_{13,1}$ & $\mu_{14,1}$ & $\nu_{2,2}$\\
\hline
$\sqrt{\rm{eigenvalue}}\approx$ & $78.48$ & $96.9$ & $97.24$ & $117.27$ & $139.58$ & $163.84$ & $190.1$ & $194.51$\\
\hline
\end{tabular}
\caption{Approximate value of the least 16 eigenvalues of \eq{eq:Eigenvalue} for $\sigma=0.25$ and $\ell=\frac{\pi}{144}$.}\label{tableigen2}
\end{center}
\end{table}

While comparing with Table \ref{tableigen}, it is noticeable that all the eigenvalues have slightly lowered but the qualitative behavior and the corresponding
explanation remain the same.

\section{Torsional stability of the longitudinal modes}\label{s:4}

\subsection{Existence, uniqueness, and finite dimensional approximation of the solution}

Let $\Upsilon$ be the characteristic function of the set
$(-\frac{\pi}{150},-\frac{3\pi}{500})\cup(\frac{3\pi}{500},\frac{\pi}{150})$
and let
\neweq{husata}
h(y,u)=\Upsilon(y)\,\Big(u+u^3\Big) \, .
\endeq
We say that
\neweq{regularity2}
u\in C^0(\R_+;H^2_*(\Omega))\cap C^1(\R_+;L^2(\Omega))\cap C^2(\R_+;\HH(\Omega))
\endeq
is a solution of \eq{wave-3} if it satisfies the initial conditions and if
\neweq{weaksolhyper}
\langle u''(t),v\rangle+\gamma
(u(t),v)_{H^2_*}+(h(y,u(t)),v)_{L^2}=0\qquad \forall v\in
H^2_*(\Omega)\, ,\ \forall t\in(0,T)\, ,
\endeq
see \eq{scalarp}. Then, by arguing as in \cite[Theorem 3.6]{fergaz} we may prove the following result.

\begin{theorem}\label{hyperbolic}
Assume \eqref{sigma}. Let $u_0\in H^2_*(\Omega)$ and $u_1\in L^2(\Omega)$. Then
there exists a unique solution $u=u(t)$ of \eqref{wave-3} and its energy \eqref{energy} satisfies
$$E(u(t))\equiv\int_\Omega\frac12 u_1^2\, dxdy+\int_\Omega\left(\frac{\gamma}{2}(\Delta u_0)^2+\frac{4\gamma}{5} ((u_0)_{xy}^2-(u_0)_{xx}(u_0)_{yy})
+\Upsilon(y)\,\left(\frac{u_0^2}{2}+\frac{u_0^4}{4}\right)\right)\, dxdy\, .$$
\end{theorem}

\noindent
{\em Sketch of the proof of Theorem \ref{hyperbolic}}. The proof of Theorem \ref{hyperbolic} makes use of a Galerkin method. The solution of
\eqref{wave-3} is the limit (in a suitable topology) of a sequence of solutions of approximated problems in finite dimensional spaces. By
Proposition \ref{eigenvalue} we may
consider an orthogonal complete system $\{w_k\}_{k\ge 1}\subset H^2_*(\Omega)$ of eigenfunctions of \eqref{eq:Eigenvalue} such that $\|w_k\|_{L^2}=1$.
Let $\{\lambda_k\}_{k\ge 1}$ be the corresponding eigenvalues and, for any $m\ge 1$, put $W_m:={\rm span }\{w_1,\dots,w_m\}$. For any $m\ge1$ let
$$u_0^m:=\sum_{i=1}^m (u_0,w_i)_{L^2}w_i =\sum_{i=1}^m\lambda_i^{-1}(u_0,w_i)_{H^2_*} \, w_i\ \mbox{ and }\ u_1^m=\sum_{i=1}^m (u_1,w_i)_{L^2} \, w_i$$
so that $u_0^m\to u_0$ in $H^2_*(\Omega)$ and $u_1^m\to u_1$ in $L^2(\Omega)$ as $m\to +\infty$. Fix $T>0$; for any $m\ge 1$ one seeks a
solution $u_m\in C^2([0,T];W_m)$ of the variational problem
\begin{equation}\label{eq:P-W-k}
\begin{cases}
(u''(t),v)_{L^2}+\gamma(u(t),v)_{H^2_*}+(h(y,u(t)),v)_{L^2}=0\\
u(0)=u_0^m\, ,\quad u'(0)=u_1^m
\end{cases}
\end{equation}
for any $v\in W_m$ and $t\in (0,T)$. If we put
\neweq{um}
u_m(t)=\sum_{i=1}^m g_i^m(t) w_i\qquad\mbox{and}\qquad g^m(t):=(g_1^m(t),\dots,g_m^m(t))^T
\endeq
then the vector valued function $g^m$ solves
\begin{equation} \label{eq:p-g-k}
\begin{cases}
(g^m(t))''+\gamma\,\Lambda_m g^m(t)+\Phi_m(g^m(t))=0 \qquad\forall t\in(0,T) \\
g^m(0)=((u_0,w_1)_{L^2},\dots,(u_0,w_m)_{L^2})^T \, ,\quad (g^m)'(0)=((u_1,w_1)_{L^2},\dots,(u_1,w_m)_{L^2})^T
\end{cases}
\end{equation}
where $\Lambda_m:={\rm diag}(\lambda_1,\dots,\lambda_m)$ and $\Phi_m:\R^m\to \R^m$ is the map defined by
$$
\Phi_m(\xi_1,\dots,\xi_m):=\left(\Big(h\Big(y,\sum_{j=1}^m \xi_j w_j\Big),w_1\Big)_{L^2},\dots,
\Big(h\Big(y,\sum_{j=1}^m \xi_j w_j\Big),w_m\Big)_{L^2}\right)^T\, .
$$
{From} \eq{husata} we deduce that $\Phi_m\in {\rm
Lip_{loc}}(\R^m;\R^m)$ and hence \eqref{eq:p-g-k} admits a unique
solution. Whence, the function $u_m(t)$ in \eq{um} belongs
$C^2([0,T);H^2_*(\Omega))$ and is a solution of the problem
\begin{equation} \label{eq:Prob-k}
\begin{cases}
u_m''(t)+\gamma\,Lu_m(t)+P_m(h(y,u_m(t)))=0 \qquad \text{for any } t\ge0\\
u_m(0)=u_0^m \, ,\quad u_m'(0)=u_1^m
\end{cases}
\end{equation}
where $L:H^2_*(\Omega)\to \mathcal H(\Omega)$ is implicitly defined by $\langle Lu,v\rangle:=(u,v)_{H^2_*}$ for any $u,v\in H^2_*(\Omega)$,
and $P_m$ is the orthogonal projection from $H^2_*(\Omega)$ onto $W_m$. By arguing as in \cite{fergaz}, one finds that the sequence $\{u_m\}$
converges in $C^0([0,T];H^2_*(\Omega))\cap C^1([0,T];L^2(\Omega))$ to a solution of \eqref{wave-3}.\qquad$\Box$\par\medskip

The above proof shows that the
solution of \eq{wave-3} may be obtained as the limit of a finite dimensional analysis performed with a finite number of modes. Let us fix some value $E>0$ for
\eq{energy}. Depending on the value of $E$, higher modes may be dropped, see \cite[Section 3.3]{bergaz} for a detailed physical motivation and for some
quantitative estimates. Our purpose is to study the torsional stability of the low modes in a sense that will be made precise in next section.
In particular, the results obtained in Section \ref{modes} suggest to focus the attention on the lowest $16$ modes, including the two least torsional modes.

\subsection{A theoretical characterization of torsional stability} \label{s:4.2}

In this section we give a precise definition of torsional
stability. We point out that this might not be the only possible
definition. However, the numerical results reported in Sections
\ref{numres} and \ref{full} show that our characterization well
describes the instability.\par We take again $h(y,u)$ as in
\eqref{husata} and we fix $m=16$ which is the position of the
second torsional mode. From Proposition \ref{eigenvalue} and Table
\ref{tableigen} we know that the eigenvalues and the
$L^2$-normalized eigenfunctions up to the 16th are given by
$$
\lambda_k=\left \{ \begin{array}{llll}
\!\!\mu_{k,1} & \text{if } 1\leq k\leq 10\\
\!\!\nu_{1,2} & \text{if } k=11\\
\!\!\mu_{k-1,1} & \text{if } 12\leq k\leq 15\\
\!\!\nu_{2,2} & \text{if } k=16
\end{array}
\right.
\qquad\qquad
w_k(x,y)=\left \{ \begin{array}{llll}
\!\! \tfrac{ v_k(y)\sin (kx)} {\omega_k }  & \mbox{if } 1\leq k\leq 10\\
\!\! \tfrac{ \theta_1(y)\sin(x)} {\bar \omega_1}  & \mbox{if } k=11\\
\!\!  \tfrac{v_{k-1}(y)\sin((k-1)x)} {\omega_{k-1}}  & \text{if } 12\leq k\leq 15\\
\!\!  \tfrac{\theta_2(y)\sin(2x)} {\bar \omega_2}  & \text{if } k=16
\end{array}\right.
$$
with
$$v_k(y):=\Big[\frac{k^2}{5}-\beta_k^{-}\Big]\, \tfrac{\cosh\Big(y\sqrt{\beta_k^{+}}\Big)}{\cosh\Big(\frac{\pi}{150}\sqrt{\beta_k^{+}}\Big)}+
\Big[\beta_k^{+}-\frac{k^2}{5}\Big]\, \tfrac{\cosh\Big(y\sqrt{\beta_k^{-}}\Big)}{\cosh\Big(\frac{\pi}{150}\sqrt{\beta_k^{-}}\Big)}\qquad
(k=1,...,14)\,,$$
$$\theta_k(y):=\Big[\frac{k^2}{5}+\alpha_{k}^{-} \Big]\, \tfrac{\sinh\Big(y\sqrt{\alpha_{k}^{+}}\Big)}{\sinh\Big(\frac{\pi}{150}\sqrt{\alpha_{k}^{+}}\Big)}+
\Big[\alpha_{k}^{+}-\frac{k^2}{5} \Big]\, \tfrac{\sin\Big(y\sqrt{\alpha_{k}^{-}}\Big)}{\sin\Big(\frac{\pi}{150}\sqrt{\alpha_{k}^{-}}\Big)}
\qquad(k=1,2)\,.$$
where $\beta_k^{\pm}:=k^2\pm \mu_{k,1}^{1/2}$ (for $k=1,...,14$), $\alpha_{k}^{\pm}:=\nu_{k,2}^{1/2}\pm k^2$ (for $k=1,2$) and
\neweq{omegakk}
\omega_k^2=\pi\int_0^{\frac{\pi}{150}}v_k^2(y)\,dy\qquad(k=1,...,14)\quad,\quad \bar \omega_k^2=\pi\int_0^{\frac{\pi}{150}}\theta_k^2(y)\,dy\qquad(k=1,2)\,.
\endeq
Notice that the $v_k$ are even with respect to $y$, while $\theta_1$ and $\theta_2$ are odd.\par
Following the Galerkin procedure described in the proof of Theorem \ref{hyperbolic}, we seek solutions of \eq{eq:P-W-k} in the form
$$u(x,y,t)=\sum_{k=1}^{14}\varphi_k(t)\, \frac{v_k(y)\sin(kx)}{\omega_k}+\sum_{k=1}^{2}\tau_k(t)\, \frac{\theta_k(y)\sin(kx)}{\bar \omega_k}$$
where the functions $\varphi_k$ and $\tau_k$ are to be determined. Take $h$ as in \eq{husata} and, for all $(\varphi_1,...,\varphi_{14},\tau_1,\tau_2)\in\R^{16}$, put
$$\Phi_k(\varphi_1,...,\varphi_{14},\tau_1,\tau_2)=\Big(h\Big(y,\sum_{j=1}^{14}\varphi_j \tfrac{ v_j(y)\sin(jx) } {\omega_j } +\sum_{j=1}^{2}\tau_j  \tfrac{\theta_j(y)\sin(jx) } {\bar \omega_j } \Big),\tfrac{ v_k(y)\sin(kx) } {\omega_k }  \Big)_{L^2}\quad(k=1,...,14)$$
$$\Gamma_k(\varphi_1,...,\varphi_{14},\tau_1,\tau_2)=\Big(h\Big(y,\sum_{j=1}^{14}\varphi_j \tfrac{v_j(y)\sin(jx)} {\omega_j } +\sum_{j=1}^{2}\tau_j \tfrac{\theta_j(y)\sin(jx)} {\bar \omega_j } \Big),
\tfrac{ \theta_k(y)\sin(kx)} {\bar \omega_k } \Big)_{L^2} \quad (k=1,2)\, .$$
Then \eq{eq:Prob-k} becomes the system of ODE's:
\begin{equation} \label{odedifficile2}
\left \{
\begin{array}{ll}
\varphi_k''(t)+\gamma\,\mu_{k,1} \varphi_k(t)+\Phi_k\big(\varphi_1(t),...,\varphi_{14}(t),\tau_1(t),\tau_2(t)\big)=0 & (k=1,...,14)\\
\tau_k''(t)+\gamma\,\nu_{k,2} \tau_k(t)+\Gamma_k\big(\varphi_1(t),...,\varphi_{14}(t),\tau_1(t),\tau_2(t)\big)=0 & (k=1,2)
\end{array}
\right.
\end{equation}
for all $t\in(0,T)$. For all $1\leq k\leq 14$ we put $\Psi_k\big(\varphi_k(t)\big)=\Phi_k(0,..., \varphi_k(t),...,0)$. By taking into account that
$$\int_0^\pi \sin^2(kx)\, dx=\frac{\pi}{2}\ ,\qquad\int_0^\pi \sin^4(kx)\, dx=\frac{3\pi}{8}\ ,\qquad\forall k\ge1\ ,$$
and that $v_k$ is even with respect to $y$, some computations yield
\begin{equation} \label{psik}
\Psi_k\big(\varphi_k(t)\big)= a_k\varphi_k(t)+b_k \varphi_k^3(t)\,,
\end{equation}
where
\neweq{akbk}
a_k= \frac{\pi}{\omega_k^2}\, \int_{\frac{3\pi}{500}}^{\frac{\pi}{150}} v_k^2(y)\,dy \quad \text{and} \quad
b_k= \frac{3\pi}{4 \omega_k^4}\,\int_{\frac{3\pi}{500}}^{\frac{\pi}{150}} v_k^4(y)\,dy\qquad(k=1,...,14)\,.
\endeq
In particular, by combining \eq{akbk} with \eq{omegakk} we see that
$$
a_k=\frac{\|v_k\|_{L^2(\frac{3\pi}{500},\frac{\pi}{150})}^2}{\|v_k\|_{L^2(0,\frac{\pi}{150})}^2}<1\quad ,\quad
b_k=\frac{3}{4\pi}\, \frac{\|v_k\|_{L^4(\frac{3\pi}{500},\frac{\pi}{150})}^4}{\|v_k\|_{L^2(0,\frac{\pi}{150})}^4}\, .
$$

We may now define what we mean by longitudinal mode. We point out that this is a classical definition in a linear regime while it is by no means standard
how to characterize modes in nonlinear regimes; contrary to the linear case, the frequency of a nonlinear mode depends on the energy or,
equivalently, on the amplitude of its oscillations.

\begin{definition}\label{oscillmodebaro}
Let $1\leq k\leq 14$, $\R^2\ni(\phi_0^k,\phi_1^k)\neq(0,0)$ and $\Psi_k$ as in \eqref{psik}.
We call $k$-th \textbf{longitudinal mode} at energy $E(\phi_0^k,\phi_1^k)>0$ the unique (periodic) solution $\overline\varphi_k$ of the Cauchy problem:
\begin{equation} \label{ode}
\begin{cases}
\varphi''_k(t)+\gamma \mu_{k,1} \varphi_k(t)+\Psi_k(\varphi_k(t))=0 \qquad\forall t>0 \\
\varphi_k(0)=\phi_0^k\, ,\quad \varphi_k'(0)=\phi_1^k\,.
\end{cases}
\end{equation}
\end{definition}

If it were $\Psi_k\equiv0$ then the equation \eq{ode} would be linear and Definition \ref{oscillmodebaro} would coincide with the usual one: in this case, there
would be no need to emphasize the dependence on the energy since the solution with initial data $\varphi_k(0)=\alpha\phi_0^k$ and
$\varphi_k'(0)=\alpha\phi_1^k$ (for any $\alpha$) would coincide with the solution of \eq{ode} multiplied by $\alpha$.
In view of \eq{psik}, we have instead a nonlinear equation and \eq{ode} becomes
\begin{equation}\label{ode1}
\begin{cases}
\varphi''_k(t)+(\gamma \mu_{k,1}+a_k) \varphi_k(t)+b_k \varphi_k^3(t)=0 \qquad\forall t>0 \\
\varphi_k(0)=\phi_0^k\, ,\quad \varphi_k'(0)=\phi_1^k\, .
\end{cases}
\end{equation}
The system \eq{ode1} admits the conserved quantity
\neweq{energyEk}
E=\frac{(\varphi'_k)^2}{2}+(\gamma \mu_{k,1}+a_k)\frac{\varphi_k^2}{2}+b_k\frac{\varphi_k^4}{4}\equiv E(\phi_0^k,\phi_1^k)=
\frac{(\phi_1^k)^2}{2}+(\gamma \mu_{k,1}+a_k)\frac{(\phi_0^k)^2}{2}+b_k\frac{(\phi_0^k)^4}{4}\,.
\endeq
Any couple of initial data having the same energy leads to the same solution of \eq{ode1} up to a time translation while it is no longer true that
multiplying the initial data by a constant leads to proportional solutions; it is well-known that different energies yield different frequencies of
the solution, see also \eq{period} below.\par
In order to define the torsional stability of a longitudinal mode $\overline\varphi_k$, we linearize the last two equations of system \eqref{odedifficile2} around
$(0,...,\overline\varphi_k(t),...,0)\in\R^{16}$. These two equations correspond, respectively, to the first and second torsional mode. In both cases we obtain
a Hill equation of the type
\neweq{Hill}
\xi''(t)+A_{l,k}(t) \xi(t)=0\,,
\endeq
where, for every $1\leq k\leq 14$ and $l=1,2$, we set
\neweq{Alk}
A_{l,k}(t)=\gamma \nu_{l,2}+\bar a_{l}+ d_{l,k}\overline \varphi_k^2(t)
\endeq
with
\neweq{aldlk}
\bar a_l=\frac{\pi}{\bar \omega_l^2} \int_{\frac{3\pi}{500}}^{\frac{\pi}{150}} \theta_{l}^2(y)\,dy
=\frac{\|\theta_l\|_{L^2(\frac{3\pi}{500},\frac{\pi}{150})}^2}{\|\theta_l\|_{L^2(0,\frac{\pi}{150})}^2}<1\, ,\qquad
d_{l,k}= \left \{
\begin{array}{lll}
\frac{9\pi}{4 \omega_l^2 \bar \omega_l^2} \int_{\frac{3\pi}{500}}^{\frac{\pi}{150}} v_l^2(y)\theta_l^2(y)\,dy  & \mbox{if } l=k\\
\frac{3\pi}{2 \omega_k^2 \bar \omega_l^2} \int_{\frac{3\pi}{500}}^{\frac{\pi}{150}} v_k^2(y)\theta_l^2(y)\,dy & \mbox{if } l\neq k\,.
\end{array}
\right.
\endeq
For the computation of these coefficients we used the facts that the integrals containing odd powers of $\theta_l(y)$ vanish and that
$$
\int_0^\pi \sin^2(lx)\sin^2(kx)\, dx=\left\{\begin{array}{ll}
\frac{3\pi}8\ & \mbox{ if }l= k\\
\frac{\pi}4\ & \mbox{ if }l\neq k\, .\end{array}\right.
$$

Since \eq{Hill} is a {\em linear} equation with periodic
coefficients the notion of stability of its trivial solution is
standard. This enables us to define the torsional stability of a
longitudinal mode.

\begin{definition}\label{newdeff}
Fix $1\leq k\leq 14$ and $l=1,2$. We say that the $k$-th longitudinal mode $\overline \varphi_k$ at energy $E(\phi_0^k,\phi_1^k)$, namely the unique periodic
solution of \eqref{ode1}, is {\bf stable with respect to the $l$-th torsional mode} if the trivial solution of \eqref{Hill} is stable.
\end{definition}

\subsection{Sufficient conditions for the torsional stability}\label{secsuffcond}

It is well-known that the stability regions for the Hill equations may have strange shapes such as {\em pockets} and {\em resonance tongues}, see
e.g.\ \cite{broer,broer2}. Therefore, the theoretical stability analysis of any such equation has to deal with these shapes and with the lack
of a precise characterization of the stability regions. For \eq{Hill}, the theoretical obstruction is essentially related to the following condition
\neweq{strange}
\sqrt{\frac{\gamma \nu_{l,2}+\bar a_{l}}{\gamma \mu_{k,1}+ a_{k}} }\not\in\N
\endeq
where $\gamma$ is defined in \eq{gamma} while $a_k$ and $\bar a_{l}$ are defined, respectively, in \eq{akbk} and in \eq{aldlk}.
The usual difficulties are further increased for \eq{Hill} which, instead of a single equation, represents {\em a family of Hill equations having coefficients with
periods depending on the energy} of the solution of \eq{ode1}.\par
Below we discuss in some detail assumption \eq{strange}. But let us start the stability analysis with the following sufficient condition
for the stability of a longitudinal mode.

\begin{theorem}\label{stable}
Fix $1\leq k\leq 14$, $l\in\{1,2\}$ and assume that \eqref{strange} holds. Then there exists $E_k^l>0$ and a strictly increasing function $\Lambda$
such that $\Lambda(0)=0$ and such that the $k$-th longitudinal mode $\overline \varphi_k$ at energy $E(\phi_0^k,\phi_1^k)$
(that is, the solution of \eqref{ode1}) is stable with respect to the $l$-th torsional mode provided that
$$E\le E_k^l\, $$
or, equivalently, provided that
$$\|\overline \varphi_k\|_\infty^2\le \Lambda(E_k^l)\,.$$
\end{theorem}

Theorem \ref{stable} {\em is not} a perturbation result. The proof given in Section \ref{stable proof} allows us to determine {\em explicit} values
of $E_k^l$ and $\Lambda(E_k^l)$. Here we stated Theorem \ref{stable} in a qualitative form in order not to spoil the statement with too many constants.
We refer to Theorem \ref{stableprecise} for the precise value of $E_k^l$ and $\Lambda(E_k^l)$.\par
Assumption \eq{strange} is a generic assumption, it has probability 1 to occur among all random choices of the positive real numbers $\gamma$, $\nu_{l,2}$,
$\bar a_{l}$, $\mu_{k,1}$, $a_{k}$. But even in the case where \eq{strange} fails we may obtain a sufficient condition for the torsional stability of
longitudinal modes.

\begin{theorem}\label{stable2}
Fix $1\leq k\leq 14$, $l\in\{1,2\}$ and assume that there exists $m\in\N$ such that
\neweq{strange2}
\frac{\gamma \nu_{l,2}+\bar a_{l}}{\gamma \mu_{k,1}+ a_{k}}=(m+1)^2\, .
\endeq
Assume moreover that
\neweq{strange3}
2\Big(2+(m+1)\pi\Big)d_{l,k}<3\pi(m+1)^3b_k\, .
\endeq
Then the same conclusions of Theorem \ref{stable} hold.
\end{theorem}

Theorem \ref{stable2}, which is proved in Section \ref{secondproof}, raises the attention on the further technical assumption \eq{strange3}.
We are confident that it might be weakened (see Remark \ref{several} at the end of Section \ref{secondproof}) and, perhaps, completely
removed (see \cite{bgz}). The reason is that, although they are not explicitly known, the stability regions for the Hill equations
are very precise and, for the model problem \eq{Hill}, there is an additional energy parameter which could vary the stability
regions. However, we will not discuss \eq{strange3} here.\par
Theorems \ref{stable} and \ref{stable2} state that a crucial role is played by the amount of energy inside the system. In the next sections we
make some numerical experiments on the equations \eq{Hill} and we study the stability of the least $14$ longitudinal modes of the plate with the TNB
parameters. Our results show that for each longitudinal mode there exists a critical energy threshold $E_k^l$ under which the solution of \eq{Hill}
is stable while for larger energies the solution may be unstable: we also show that different initial data with the same total energy give the same
stability response. Not only this enables us to numerically compute the threshold $E_k^l$ and to evaluate the power of the sufficient condition given in
Theorems \ref{stable} and \ref{stable2}, but also to define a flutter energy for each longitudinal mode as a threshold of stability.

\begin{definition}\label{flutter} We call \textbf{flutter energy} of the $k$-th longitudinal mode $\overline \varphi_k$ (that is, the solution of \eqref{ode})
the positive number $\overline{E}_k$ being the supremum of the energies $E_k^l$ such that the trivial solution of \eqref{Hill} is stable for both $l=1$ and $l=2$.
\end{definition}

Concerning the bridge model, Theorems \ref{stable} and \ref{stable2} lead to the conclusion that\par\noindent
{\em if the internal energy $E$ is smaller than the flutter energy then small initial torsional oscillations remain small for all time $t>0$,
whereas if $E$ is larger than the flutter energy (that is, the longitudinal oscillations are initially large) then small torsional oscillations
may suddenly become wider}.

\section{Numerical computation of the flutter energy}\label{numres}

First, by using the eigenvalues found in Table \ref{tableigen}, we numerically compute the parameters $a_k$ and $b_k$ in \eq{akbk}. It turns out that
all the $a_k$ are equal to $0.1$ up to an error of less than $10^{-3}$: for this reason, in Table \ref{ak01} we quote the values of $10^4(a_k-0.1)$. Moreover,
all the $b_k$ are equal to $1.1$ up to an error of less than $10^{-1}$: we quote the values of $10^2(b_k-1.1)$.

\begin{table}[ht]
\begin{center}
\begin{tabular}{|c|c|c|c|c|c|c|c|c|c|c|c|c|c|c|}
\hline
$k$ & 1 & 2 & 3 & 4 & 5 & 6 & 7 & 8 & 9 & 10 & 11 & 12 & 13 & 14 \\
\hline
$10^4(a_k\!-\!0.1)$ & 0.05 & 0.2 & 0.45 & 0.8 & 1.24 & 1.78 & 2.42 & 3.14 & 3.96 & 4.86 & 5.85 & 6.92 & 8.06 & 9.28 \\
\hline
$10^2(b_k\!-\!1.1)$ & 4 & 4.03 & 4.09 & 4.17 & 4.27 & 4.39 & 4.54 & 4.7 & 4.89 & 5.1 & 5.32 & 5.57 & 5.83 & 6.11 \\
\hline
\end{tabular}
\caption{Numerical values of $a_k$ and $b_k$.}\label{ak01}
\end{center}
\end{table}

Then we compute $\bar a_l$ and $d_{l,k}$ as defined in \eq{aldlk}. We find that both $\bar a_1\approx0.27$, $\bar a_2\approx0.27$, while the $d_{l,k}$
are as reported in Table \ref{tabdlk}.

\begin{table}[ht]
\begin{center}
\begin{tabular}{|c|c|c|c|c|c|c|c|c|c|c|c|c|c|c|}
\hline
$k$ & 1 & 2 & 3 & 4 & 5 & 6 & 7 & 8 & 9 & 10 & 11 & 12 & 13 & 14 \\
\hline
$d_{1,k}$ & 9.27 & 6.18 & 6.18 & 6.18 & 6.19 & 6.19 & 6.19 & 6.2 & 6.2 & 6.21 & 6.21 & 6.22 & 6.23 & 6.24 \\
\hline
$d_{2,k}$ & 6.18 & 9.27 & 6.18 & 6.18 & 6.19 & 6.19 & 6.19 & 6.2 & 6.2 & 6.21 & 6.21 & 6.22 & 6.23 & 6.24 \\
\hline
\end{tabular}
\caption{Numerical values of $d_{1,k}$ and $d_{2,k}$.}\label{tabdlk}
\end{center}
\end{table}

In fact, $0<d_{1,k}-d_{2,k}\approx10^{-4}$ (for $k=3,...,14$) so
that, in Table \ref{tabdlk}, one does not see any difference
between these coefficients: we put however the ``exact'' numerical
values in the below numerical experiments.\par Concerning the
structural parameters and the value of $\gamma$ we first notice
that the total length of the composed spring depicted in Figure
\ref{cableshangers} equals the height of the towers over the
roadway, which is $72\, m$, see \cite{ammann}. We denote by $h$
the sag of the parabola drawn by the cables: since the shortest
hangers at midspan were of $2\, m$ we have $h=70 \, m$. We denote
by $H$ and $V$ respectively the horizontal and vertical components
of the tension of the couple of cables; in the equilibrium
position under the action of the weight of the cables, of the
hangers and of the deck, the tension is given by the vector
$$
(H,V)=|(H,V)|
\overrightarrow{\tau}=\frac{|(H,V)|}{\sqrt{1+\left[\frac{4h}{L^2}(L-2x)\right]^2}}\left(1,\frac{4h}{L^2}(L-2x)\right)\, .
$$
where $\overrightarrow{\tau}$ denotes the tangent unit vector to the parabola. From this we deduce that
$$
\frac{|(H,V)|}{\sqrt{1+\left[\frac{4h}{L^2}(L-2x)\right]^2}}=H \,
,\qquad
V=\frac{|(H,V)|}{\sqrt{1+\left[\frac{4h}{L^2}(L-2x)\right]^2}}\frac{4h}{L^2}(L-2x)=H\frac{4h}{L^2}(L-2x)
\, .
$$
The horizontal component of the tension $H$ can be supposed constant with respect to $x$ and hence we may write
$$
V(x)=H\frac{4h}{L^2}(L-2x) \qquad \text{for any } x\in (0,L) \, .
$$
Therefore, the action of the cables per unit of length is given by
$$
V'(x)=-\frac{8H}{L^2}h \qquad \text{for any } x\in (0,L) \, .
$$
For more details on the behavior of the cables and how to derive
the above formulas we address the interested readers to
\cite[Chapter 3]{lacarbonara}.\par If we consider a configuration
corresponding to a displacement $w$ from the equilibrium
configuration, the increment of the action of the cables per unit
of length is given by $\frac{8H}{L^2}w$. Taking into account that
the weight of the bridge per unit of length is approximatively
$83\, kN/m$, we have that $H=1.08\cdot 10^8\, N$.

We have seen in Section \ref{nonmod} that the coupled action of cables and hangers is nonlinear. Following \eq{gu} we suggest as
possible force the quantity
$$
\frac{8H}{L^2}(w+w^3)\, .
$$

Next we compute the action of the cables and of the hangers per unit of surface: since the hangers act on
the region $\omega$ defined in \eqref{omega} we obtain
$$
\frac{1500}{2L}\, \frac{8H}{L^2}(w+w^3)=\frac{6000H}{L^3}(w+w^3)\, .
$$
Taking into account that $m=635 \ kg/m^2$ we come to the equation
$$
m\, w_{tt}+\Gamma \Delta^2 w+\frac{6000H}{L^3}\Upsilon(y)(w+w^3)=0
$$
where $\Gamma$ denotes the rigidity of the plate.

In order to provide a reasonable value for the
rigidity compatible with the parameters of the TNB, we start by considering the rigidity of the deck in the beam model. The rigidity of a beam
is given by $EI$ where $E$ is the Young modulus and $I$ is the moment of inertia of the cross section with respect to the
horizontal axis orthogonal to the axis of the beam and containing its barycenter; from \cite{ammann} we know that
$E=2.1\cdot 10^{11} \, Pa$ and $I=0.1528 \, m^4$. Then, the plate equivalent to the beam has a rigidity given by
$\Gamma=\frac{EI}{2\ell(1-\sigma^2)}=2.937 \cdot 10^9 \ Pa \cdot m^3$, see \cite{fergaz} for more details on the comparison between the two models.
From \cite{fergaz} we also recall that the rigidity of a plate is given by $\frac{Ed^3}{12(1-\sigma^2)}$ where $d$ is the thickness of the plate. In our case
we can recover $d$ from the value of $\Gamma$: indeed we have $d=[12(1-\sigma^2)\Gamma/E]^{1/3}=0.544 \ m$. We observe that this is not the thickness of
the roadway slab, which is known to be of $13\ cm$, see \cite[p.13]{ammann}. However, since the roadway slab was reinforced by a stiffening girder with an
H-shaped section, the deck may be considered as a plate of thickness four times as much: $d\approx52\ cm$. Finally, we recall that a rectangular plate has
to be considered as a rectangular parallelepiped of thickness $d$ made of a homogeneous material with constant density and isotropic behavior.\par
We have so determined the constants in the model problem \eqref{tac} and, defining $u$ as in \eqref{gamma}, we come to the adimensional problem
\eqref{wave-3} where
\neweq{truegamma}
\gamma=\frac{\pi^4\Gamma}{6000HL}=5.17\cdot 10^{-4}.
\endeq

Since we aim to show that the value of $\gamma$ plays a secondary role, we quote below the results for three different
$\gamma$.\par Once all these parameters are fixed, we solve \eq{ode1} for $\varphi_k(0)=A>0$ and $\varphi_k'(0)=0$ for
different values of $A$. Each value of $A$ yields the $k$-th longitudinal mode $\overline \varphi_k=\overline \varphi_k^A$ at
energy $E(A,0)>0$, see Definition \ref{oscillmodebaro}. In turn,
we use $\overline \varphi_k^A$ to compute the function
$A_{l,k}(t)$ defined in \eq{Alk} and we replace it into \eq{Hill}.
We start with $A=0$ and we increase it until the trivial solution
$\xi_0\equiv0$ of \eq{Hill} becomes unstable. Since this is a very
delicate point, let us explain with great precision how we obtain
the two sets of critical values of $A$ (thresholds of instability)
that we denote by $A_1(k)$ and $A_2(k)$.\par For a particular
second order Hamiltonian system of the kind of \eq{odedifficile2}, it is shown in \cite{bgz} that there exist two
increasing and divergent sequences $\{A_l^n\}_{n=0}^\infty$
($l=1,2$) such that:\par\noindent $\bullet$ if $A\in
S:=\cup_k(A_l^{2k},A_l^{2k+1})$ then $\xi_0$ is
stable;\par\noindent $\bullet$ if $A\in
U:=\cup_k(A_l^{2k+1},A_l^{2k+2})$ then $\xi_0$ is
unstable.\par\noindent Moreover, the instability becomes more
evident if $A$ is far from $S$: in particular, if
$A\in(A_l^{2k+1},A_l^{2k+2})$ for some $k\ge0$ and
$A_l^{2k+2}-A_l^{2k+1}>0$ is small, then it is hard to detect the
instability of $\xi_0$.\par Due to the already mentioned
unpredictable behavior of the stability regions for general Hill
equations (see \cite{broer,broer2}), the same results appear
difficult to reach for the particular Hamiltonian system \eq{odedifficile2}. It is however reasonable to expect that somehow
similar results hold. In particular, we expect $\xi_0$ to be
``weakly unstable'' whenever $A$ belongs to a narrow interval of
instability. By this we mean that nontrivial solutions of
\eq{Hill} blow up slowly in time. In other words, if $A$ belongs
to a narrow instability interval, then only a small amount of
energy is transferred from the longitudinal mode $\overline
\varphi_k^A$ to a torsional mode. From a physical point of view,
this kind of instability is irrelevant, both because it has low
probability to occur and because, even if it occurs, the torsional
mode remains fairly small. In turn, from a mechanical point of
view, we know from Scanlan-Tomko \cite{scantom} that small
torsional oscillations are harmless and the bridge would remain
safe. For this reason, we compute the two sets of critical values
$A_l(k)$ ($l=1,2$) as the infimum of the first interval of
instability having at least amplitude 0.2. These critical values
measure the least height of the longitudinal mode
$\overline\varphi_k$ which gives rise to a ``strong'' instability.\par
Let us explain what we saw numerically. For $\gamma=10^{-4}$ we discuss here the
stability of the $14th$ longitudinal mode $\overline
\varphi_{14}^A$ for $A\in\{0.79;0.8;0.81\}$ with respect to the
first torsional mode. In the pictures of Figure \ref{tre} we
represent the corresponding solution of \eq{Hill} with
$\xi(0)=\xi'(0)=1$.

\begin{figure}[ht]
\begin{center}
{\includegraphics[height=32mm, width=55mm]{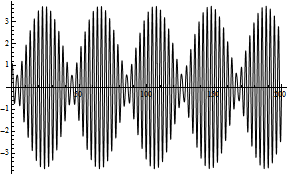}}\ {\includegraphics[height=32mm, width=55mm]{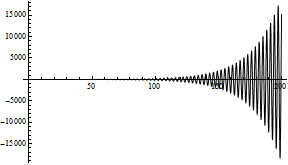}}\
{\includegraphics[height=32mm, width=55mm]{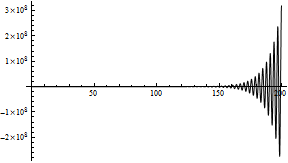}}
\caption{Solutions of \eq{Hill} with $\xi(0)=\xi'(0)=1$, $k=14$, $l=1$, $A\in\{0.79;0.8;0.81\}$.}\label{tre}
\end{center}
\end{figure}

When $A<0.75$ the behavior of $\xi$ appears periodic with, basically, oscillations of constant amplitude. When $A=0.79$ (left picture) the function $\xi$
still appears periodic but with oscillations of variable amplitude; this behavior always turned out to be the foreplay of instability. And, indeed, when
$A=0.8$ we see (middle picture) that $\xi(t)$ oscillates with increasing amplitude and reaches a magnitude of the order of $10^4$ for $t=200$. The
same phenomenon is accentuated for $A=0.81$ (right picture) where $\xi(200)$ has an order of magnitude of $10^8$. By increasing further $A$ the magnitude
was also increasing. This is how we found $A_1(14)=0.8$.\par
{From} the values of $A_1(k)$ and $A_2(k)$ one can compute the corresponding energies $E_1$ and $E_2$ as given by \eq{energyEk}:
\neweq{Elk}
E_l(k)=(\gamma \mu_{k,1}+a_k)\frac{A_l(k)^2}{2}+b_k\frac{A_l(k)^4}{4}\qquad(l=1,2)\,.
\endeq
Then we can compute the flutter energy of the $k$-th longitudinal mode $\overline \varphi_k$ (see Definition \ref{flutter}) as
$$E_k=\min\{E_1(k),E_2(k)\}\, .$$

In Table \ref{threegamma} we quote our numerical results for some values of $\gamma$. In the middle table we use \eq{truegamma}, that is, the value of $\gamma$
obtained from the parameters of the TNB. We found completely similar behaviors also for $\gamma=10^{-2}$ and $\gamma=5\cdot10^{-5}$.

\begin{table}[ht]
\begin{center}
{\small
\begin{tabular}{|c|c|c|c|c|c|c|c|c|c|c|c|c|c|c|}
\hline
$k$ & 1 & 2 & 3 & 4 & 5 & 6 & 7 & 8 & 9 & 10 & 11 & 12 & 13 & 14 \\
\hline
$A_1(k)$ & 4.62 & 2.96 & 2.93 & 2.85 & 2.67 & 2.31 & 1.42 & $>$20 & $>$20 & $>$20 & 0.89 & 1.51 & 2.02 & 2.51 \\
\hline
$A_2(k)$ & 5.93 & 9.23 & 5.91 & 5.87 & 5.79 & 5.63 & 5.36 & 4.92 & 4.19 & 2.62 & $>$20 & $>$20 & $>$20 & $>$20 \\
\hline
\end{tabular}}
\end{center}

\begin{center}
{\small
\begin{tabular}{|c|c|c|c|c|c|c|c|c|c|c|c|c|c|c|}
\hline
$k$ & 1 & 2 & 3 & 4 & 5 & 6 & 7 & 8 & 9 & 10 & 11 & 12 & 13 & 14 \\
\hline
$A_1(k)$ &  3.32 & 2.12 & 2.10 & 2.04 & 1.92 & 1.65 &  1.01  & $>20$ & $>20$ & $>20$ &  0.62  &  1.08  & 1.46   & 1.82 \\
\hline
$A_2(k)$ &  4.26 & 2.46 & 4.25 & 4.22 & 4.15 & 4.05 &  3.86  &  3.54  &  3.01 &  1.87 & $>20$ & $>20$ & $>20$ & $>20$ \\
\hline
\end{tabular}}
\end{center}

\begin{center}
{\small
\begin{tabular}{|c|c|c|c|c|c|c|c|c|c|c|c|c|c|c|}
\hline
$k$ & 1 & 2 & 3 & 4 & 5 & 6 & 7 & 8 & 9 & 10 & 11 & 12 & 13 & 14 \\
\hline
$A_1(k)$ & 1.44 & 0.91 & 0.9 & 0.88 & 0.82 & 0.69 & $>$10 & $>$10 & $>$10 & $>$10 & 0.2 & 0.44 & 0.63 & 0.8 \\
\hline
$A_2(k)$ & 1.87 & 2.91 & 1.86 & 1.85 & 1.82 & 1.77 & 1.69 & 1.54 & 1.3 & 0.76 & $>$10 & $>$10 & $>$10 & $>$10 \\
\hline
\end{tabular}}
\caption{Numerical values of $A_1(k)$ and $A_2(k)$ when $\gamma=10^{-3}$ (top), $\gamma=5.17\cdot10^{-4}$ (middle), $\gamma=10^{-4}$ (bottom).}\label{threegamma}
\end{center}
\end{table}

\begin{remark}\label{primarem}
{\rm From Table \ref{threegamma} we deduce that, if $\gamma=10^{-3}$ or $\gamma=5.17\cdot10^{-4}$ then $A_1(k)>A_2(k)$ provided that $k=8,9,10$.
If $\gamma=10^{-4}$ this happens provided that $k=7,8,9,10$. In these cases, the energy transfer occurs on the second torsional mode.
On the contrary, for lower $k$, we have that $A_1(k)<A_2(k)$.}
\end{remark}

These results deserve several comments. First of all, we notice
that the values of $A_1(1)$ and $A_2(2)$ are somehow out of the
pattern because they are strongly influenced by the definition of
the coefficients $d_{l,k}$ in \eq{aldlk} which is fairly different
if $l=k$ and $l\neq k$. If we drop the case $k=l$, from Table
\ref{threegamma} we see that the maps $k\mapsto A_l(k)$ are
strictly decreasing until some $k$ where $A_l(k)$ becomes very
large. In particular, when $l=2$ the least amplitude of
oscillation $A_2(k)$ (threshold of instability) is obtained for
$k=10$. This behavior is obtained for different values of
$\gamma$. As pointed out in the Introduction, on the day of the
TNB collapse the motions {\em were considerably less than had
occurred many times before}. Table \ref{threegamma} gives an
answer to question {\bf (Q3)}: the tenth longitudinal mode seems
more prone to generate the second torsional mode.

\section{Validation of the results}\label{full}

\subsection{Validation of the linearization procedure}

Fix some $k\in\{1,...,14\}$ and some $l\in\{1,2\}$. The energies
$E_l(k)$ in \eq{Elk} are computed in Section \ref{numres} with the
following algorithm. Firstly we solve \eq{ode1} and find the
$k$-th longitudinal mode $\overline{\varphi}_k$ at energy
$E(\phi_0^k,\phi_1^k)$. Then we use $\overline{\varphi}_k$ to
determine the function $A_{l,k}(t)$ in \eq{Alk}. Finally, we study
the stability of the trivial solution of \eq{Hill}: if it is
stable then $E(\phi_0^k,\phi_1^k)<E_l(k)$ while if it is unstable,
then $E(\phi_0^k,\phi_1^k)>E_l(k)$. With different choices of the
initial data $(\phi_0^k,\phi_1^k)$ we are so able to find both
upper and lower bounds which can be as close as wanted.\par
Overall, this algorithm means that we are solving the following
system
\begin{equation} \label{ode11}
\begin{cases}
\varphi''_k(t)+(\gamma \mu_{k,1}+a_k) \varphi_k(t)+b_k \varphi_k^3(t)=0 \qquad\forall t>0 \\
\xi_l''(t)+(\gamma \nu_{l,2}+\bar a_{l}+ d_{l,k}\varphi_k^2(t))\xi_l(t)=0\qquad\forall t>0 \\
\varphi_k(0)=\phi_0^k\, ,\quad \varphi_k'(0)=\phi_1^k\\
\xi_l(0)=\eps_0\, ,\quad \xi_l'(0)=\eps_1
\end{cases}
\end{equation}
where $|\eps_1|+|\eps_2|\ll|\phi_0^k|+|\phi_1^k|$. The system \eq{ode11} is obtained by putting to 0 all the $\varphi_j$ with $j\neq k$ and $\tau_j$
with $j\neq l$ in \eq{odedifficile2} and by linearizing the so found $2\times2$ system of ODE's in a neighborhood of the solution
$(\overline{\varphi}_k,0)$.\par
One may wonder if this method to determine $E_l(k)$ does not suffer from this linearization and give incorrect results.
We have so tried a nonlinear approach and considered systems such as
\begin{equation} \label{ode12}
\begin{cases}
\varphi''_k(t)+(\gamma \mu_{k,1}+a_k) \varphi_k(t)+b_k \varphi_k^3(t)+\alpha_1\xi_l^2(t)\varphi_k(t)+\alpha_2\xi_l(t)\varphi_k^2(t)=0 \qquad\forall t>0 \\
\xi_l''(t)+(\gamma \nu_{l,2}+\bar a_{l}+ d_{l,k}\varphi_k^2(t))\xi_l(t)+\beta_1\varphi_k(t)\xi_l^2(t)+\beta_2\xi_l^3(t)=0\qquad\forall t>0 \\
\varphi_k(0)=\phi_0^k\, ,\quad \varphi_k'(0)=\phi_1^k\\
\xi_l(0)=\eps_0\, ,\quad \xi_l'(0)=\eps_1\, ;
\end{cases}
\end{equation}
note that the equations in \eq{ode12} differ from the ones in
\eq{ode11} by third order homogeneous polynomials with respect to
$\varphi_k$ and $\xi_l$. We numerically solved \eq{ode12} for
different choices of the constants $\alpha_i$ and $\beta_i$ but,
regardless of the choice, we always found that
\begin{center}
{\bf the solution $\xi_l$ of \eq{ode11} remains small for all $t>0$ if and only if the solution $\xi_l$ of \eq{ode12} remains small for all $t>0$.}
\end{center}
This behavior had to be expected in view of the theoretical result
in \cite{ortega}.

This means that the computation of $E_l(k)$ (and of the flutter energy of the $k$-th longitudinal mode) does not depend on the linearization
procedure. Moreover, this shows that Definition \ref{newdeff} well characterizes the stability of the $k$-th longitudinal mode with respect to
the $l$-th torsional mode.

\subsection{Different nonlinearities: a system of Hill equations}

For any longitudinal mode $k$, the just described procedure with the specific nonlinearity chosen, see \eq{husata}, yields the system \eq{Hill}
of {\em uncoupled} Hill equations ($l=1,2$). This enables us to study separately the stability of each of the two equations in \eq{Hill} and to compute
the thresholds $A_l(k)$ as described above.\par
For different nonlinearities, other than \eq{husata}, it may happen that the corresponding equations in \eq{Hill} ($l=1,2$) remain coupled and give
rise to a Hill {\em system} of the form
\neweq{Hillsystem}
\Xi''(t)+A_{l,k}(t) \Xi(t)=0
\endeq
where $\Xi=(\xi_1,\xi_2)$ and $A_{l,k}$ is now a $2\times2$ periodic matrix. In this case, one should investigate the stability of the trivial solution
$\Xi\equiv(0,0)$ of \eq{Hillsystem}. By \cite[Theorem II p.270, vol.1]{yakubovich} we know that this trivial solution is stable provided that the
trivial solutions of a family of related Hill equations are stable. Therefore, even for different nonlinearities or for a larger number of torsional
modes, the stability of the $k$-th longitudinal mode with respect to the torsional components may be performed through a finite number of stability
analysis of some Hill {\em equations}.

\subsection{The coupled nonlinear system}\label{truncsys}

In this section we discuss the nonlinear system \eqref{odedifficile2} with $\gamma=5.17\cdot 10^{-4}$ as for the TNB, see \eq{truegamma}. We provide an
alternative notion of stability for a solution of \eqref{odedifficile2} having the torsional components $\tau_1$ and $\tau_2$ identically equal to zero.
Then we discuss the stability of the longitudinal components. Our purpose is to explain to which extent the decoupling procedure applied in Section
\ref{s:4.2} provides an accurate description of the phenomena.\par
For our convenience we slightly modify some notations.
We denote by $\widetilde w_1,\dots,\widetilde w_{14}$ the first fourteen eigenfunctions (see Section \ref{s:4.2}) which are even with respect to $y$, i.e.\
$\widetilde w_k=w_k$ for any index $k\in \{1,\dots,10\}$ and $\widetilde w_{11}=w_{12}$, $\widetilde w_{12}=w_{13}$,
$\widetilde w_{13}=w_{14}$, $\widetilde w_{14}=w_{15}$. Then we denote by $\widetilde w_{15}=w_{10}$ and $\widetilde
w_{16}=w_{16}$ as the first two eigenfunctions which are odd with respect to $y$. All the eigenfunctions are normalized in $L^2(\Omega)$.
Similarly we relabel the functions $\varphi_1,\dots,\varphi_{14}$ and $\tau_1,\tau_2$ as
$\widetilde \varphi_k$ for $k\in \{1,...,16\}$ and the corresponding eigenvalues $\widetilde \mu_k$ for any $k\in\{1,\dots,16\}$.
Then, for $\Upsilon$ as in \eq{husata}, we introduce the constants
\begin{equation*}
A_k=\int_\Omega \Upsilon(y)\widetilde w_k^2 \, dxdy \quad
\text{for any } k\in \{1,\dots,16\}
\end{equation*}
and for any $j_1,j_2,j_3\in \{1,\dots,16\}$, $j_1\ge j_2\ge j_3$ and $k\in \{1,\dots,16\}$
\begin{equation} \label{eq:B}
B_{j_1j_2j_3k}=
\begin{cases}
\int_\Omega \Upsilon(y)\widetilde w_{j_1}\widetilde
w_{j_2}\widetilde w_{j_3}\widetilde w_{k} \, dxdy & \quad \text{if } j_1=j_2=j_3 \\[8pt]
3\int_\Omega \Upsilon(y)\widetilde w_{j_1}\widetilde
w_{j_2}\widetilde w_{j_3}\widetilde w_{k} \, dxdy & \quad \text{if } j_1>j_2=j_3 \ \text{or} \ j_1=j_2>j_3 \\[8pt]
6\int_\Omega \Upsilon(y)\widetilde w_{j_1}\widetilde
w_{j_2}\widetilde w_{j_3}\widetilde w_{k} \, dxdy & \quad \text{if } j_1>j_2>j_3 \, .
\end{cases}
\end{equation}
With these notations we may write \eqref{odedifficile2} and the corresponding initial conditions in the form
\begin{equation} \label{eq:sistema-completo}
\begin{cases}
\ds{\widetilde\varphi_k''(t)+(\gamma\widetilde\mu_k+A_k)\widetilde\varphi_k(t)+\sum_{\underset{j_1\ge j_2\ge
j_3}{j_1,j_2,j_3=1}}^{16} B_{j_1 j_2 j_3 k}\, \widetilde
\varphi_{j_1}(t) \widetilde \varphi_{j_2}(t) \widetilde
\varphi_{j_3}(t)=0\, , } & \qquad  k\in \{1,\dots,16\} \, , \\[25pt]
\widetilde \varphi_k(0)=\widetilde \phi_0^k \, , \quad \widetilde \varphi_k'(0)=\widetilde \phi_1^k \, , & \qquad  k\in \{1,\dots,16\} \, .
\end{cases}
\end{equation}
We observe that if $k=15,16$ then $B_{j_1 j_2 j_3 k}=0$ for any $j_1,j_2,j_3\in \{1,\dots,14\}$ such that $j_1\ge j_2\ge j_3$; therefore if we choose
$\widetilde \phi_0^{15}=\widetilde \phi_1^{15}=\widetilde \phi_0^{16}=\widetilde \phi_1^{16}=0$, then the solution of \eqref{eq:sistema-completo}
satisfies $\widetilde \varphi_{15}=\widetilde \varphi_{16}\equiv 0$. This means that no torsional oscillation may appear if such a kind of oscillation
equals zero at $t=0$. We summarize this fact in the following

\begin{proposition} \label{p:uncoupled} The unique solution $(\widetilde \varphi_1,...,\widetilde \varphi_{16})$ of \eqref{eq:sistema-completo} with
$\widetilde \phi_0^{15}=\widetilde \phi_1^{15}=\widetilde \phi_0^{16}=\widetilde \phi_1^{16}=0$ satisfies $\widetilde \varphi_{15}(t)=\widetilde \varphi_{16}(t)=0$
for any $t\in \R$.
\end{proposition}

The situation is completely different if we do not assume that $\widetilde \phi_0^{15}=\widetilde \phi_1^{15}=\widetilde \phi_0^{16}=\widetilde \phi_1^{16}=0$:
in this case the first fourteen equations are coupled with the last two since the $B_{kj_2 j_3 k}$ may be different from zero if $k\in\{15,16\}$ and $j_2,j_3\le14$.
This suggests to make precise how an oscillation mode may influence the others.

\begin{definition} \label{d:influence} Let us consider system \eqref{eq:sistema-completo}.
\begin{itemize}
\item[(i)] We say that the $k$-th mode influences the $j$-th mode ($j\neq k$) if $B_{kkkj}\neq 0$.
\item[(ii)] We say that the $k_1$-th and $k_2$-th modes ($k_1>k_2$) influence the $j$-th mode ($j\not\in \{k_1,k_2\}$) if at least one between
$B_{k_1k_1k_2j}$ or $B_{k_1k_2k_2j}$ is different from zero.
\item[(iii)] We say that the $k_1$-th, $k_2$-th, $k_3$-th modes ($k_1>k_2>k_3$) influence the $j$-th mode ($j\not\in \{k_1,k_2,k_3\}$) if $B_{k_1k_2k_3j}\neq 0$.
\end{itemize}
\end{definition}

If some modes influence the $j$-th mode it may happen that, in
system \eqref{eq:sistema-completo}, $\widetilde\varphi_j\not\equiv
0$ even if $(\widetilde\phi_0^j,\widetilde\phi_1^j)=(0,0)$. We
state a result which explains how longitudinal modes influence
each other.

\begin{proposition} \label{p:influence}
The following statements hold true:
\begin{itemize}
\item[(i)] Let $k,j \in \{1,\dots,14\}$ with $j\neq k$. If $j=3k$ then the $k$-th mode influences the $j$-th mode;
\item[(ii)] Let $k_1,k_2,j \in \{1,\dots,14\}$ with $k_1>k_2$ and $j\not\in\{k_1,k_2\}$.
If $j\in \{2k_1+k_2,2k_1-k_2,k_1+2k_2,|k_1-2k_2|\}$ then the $k_1$-th and $k_2$-th modes influence the $j$-th mode.
\item[(iii)] Let $k_1,k_2,k_3,j\in \{1,\dots,14\}$ with $k_1>k_2>k_3$ and $j\not\in \{k_1,k_2,k_3\}$.
If $j \in \{k_1+k_2+k_3,k_1+k_2-k_3,k_1-k_2+k_3,|k_1-k_2-k_3|\}$ then the $k_1$-th, $k_2$-th and $k_3$-th modes influence the $j$-th mode.
\end{itemize}
\end{proposition}
\begin{proof}
In order to treat (i)-(ii) together we assume that $k_1\ge k_2$
where also equality is admissible. Then, under the assumptions of
the proposition we infer that either $B_{k_1k_1k_2j}\neq 0$ or
$B_{k_1k_2k_2j}\neq 0$ since either $\int_0^\pi
\sin(k_1x)\sin(k_1x)\sin(k_2x)\sin(jx)\,dx\neq 0$ or $\int_0^\pi
\sin(k_1x)\sin(k_2x)\sin(k_2x)\sin(jx)\,dx\neq 0$ thanks to Werner
formulas; moreover $v_k(y)>0$ for any $y\in \left(-\frac
\pi{150},\frac \pi{150}\right)$ and $k\in \{1,\dots,14\}$, with
$v_k$ as in Section \ref{s:4.2}. The case (iii) can be treated in
a similar way.
\end{proof}

\begin{remark}
{\rm Proposition \ref{p:influence} strongly depends on the nonlinearity $h$ introduced in \eqref{husata}. That the $k$-th mode influences the $j$-th
mode when $j=3k$ (statement (i)) depends on the fact that the nonlinearity $h(y,u)$ contains a cubic term. If the cubic term is replaced with a different
term then the influences may vary. A similar discussion is valid for (ii)-(iii) in Proposition \ref{p:influence}.}
\end{remark}

To illustrate Proposition \ref{p:influence}, we consider the case
where $\widetilde \phi_0^1=1$, $\widetilde \phi_1^1=0$ and
$\widetilde \phi_0^k=\widetilde \phi_1^k=0$ for all $k\in
\{2,\dots,16\}$. According to Proposition \ref{p:influence}, the
first longitudinal mode influences only the third longitudinal
mode but this does not mean that only $\widetilde\varphi_1$ and
$\widetilde\varphi_3$ are nontrivial in
\eqref{eq:sistema-completo}. As one can see from Figure
\ref{f:influence-1} below, where a numerical solution of
\eqref{eq:sistema-completo} is obtained, also the functions
$\widetilde\varphi_5,\widetilde\varphi_7,\widetilde\varphi_9,\widetilde\varphi_{11},\widetilde\varphi_{13}$
are not identically equal to zero while all the other longitudinal
modes and the two torsional modes are identically equal to zero.
\begin{figure}[ht]
\begin{center}
{\includegraphics[height=90mm, width=170mm]{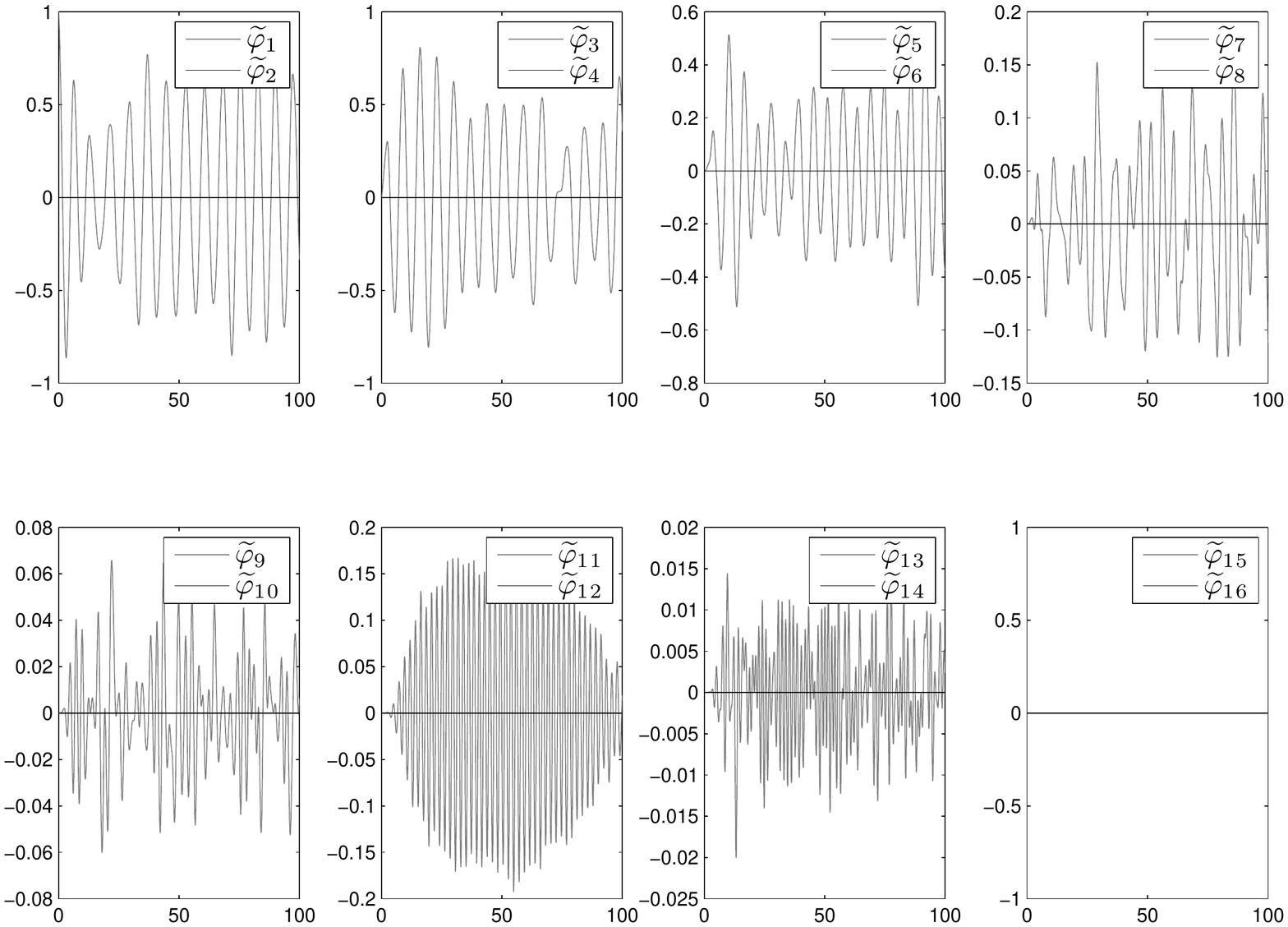}}
\caption{The solution of \eqref{eq:sistema-completo} with $\widetilde\phi_0^1=1$, $\widetilde\phi_1^1=0$, $(\widetilde\phi_0^k,\widetilde\phi_1^k)=(0,0)$ for any $k\in\{2,\dots,16\}$.}\label{f:influence-1}
\end{center}
\end{figure}

Let us briefly discuss these results.
From Proposition \ref{p:influence} (i) with $k=1$ and $j=3$, we deduce that $\widetilde\varphi_1$ influences $\widetilde\varphi_3$; by Proposition \ref{p:influence} (ii) with $k_1=3$ and $k_2=1$ we have that $\widetilde\varphi_1$ and $\widetilde\varphi_3$ influence $\widetilde\varphi_5$. Then
$\widetilde\varphi_1$ and $\widetilde\varphi_5$ influence $\widetilde\varphi_7$. In the next stages $\widetilde\varphi_1$ and $\widetilde\varphi_7$ influence $\widetilde\varphi_9$, while $\widetilde\varphi_1$ and $\widetilde\varphi_9$ influence $\widetilde\varphi_{11}$; finally $\widetilde\varphi_1$ and $\widetilde\varphi_{11}$ influence $\widetilde\varphi_{13}$.\par
As a second example we consider $\widetilde \phi_0^2=1$, $\widetilde \phi_1^2=0$ and $(\widetilde \phi_0^k,\widetilde \phi_1^k)=(0,0)$ for any $k\in \{1,\dots,16\}$
with $k\neq 2$. A similar phenomenon occurs as one can see from Figure \ref{f:influence-2} below.
\begin{figure}[ht]
\begin{center}
{\includegraphics[height=90mm, width=170mm]{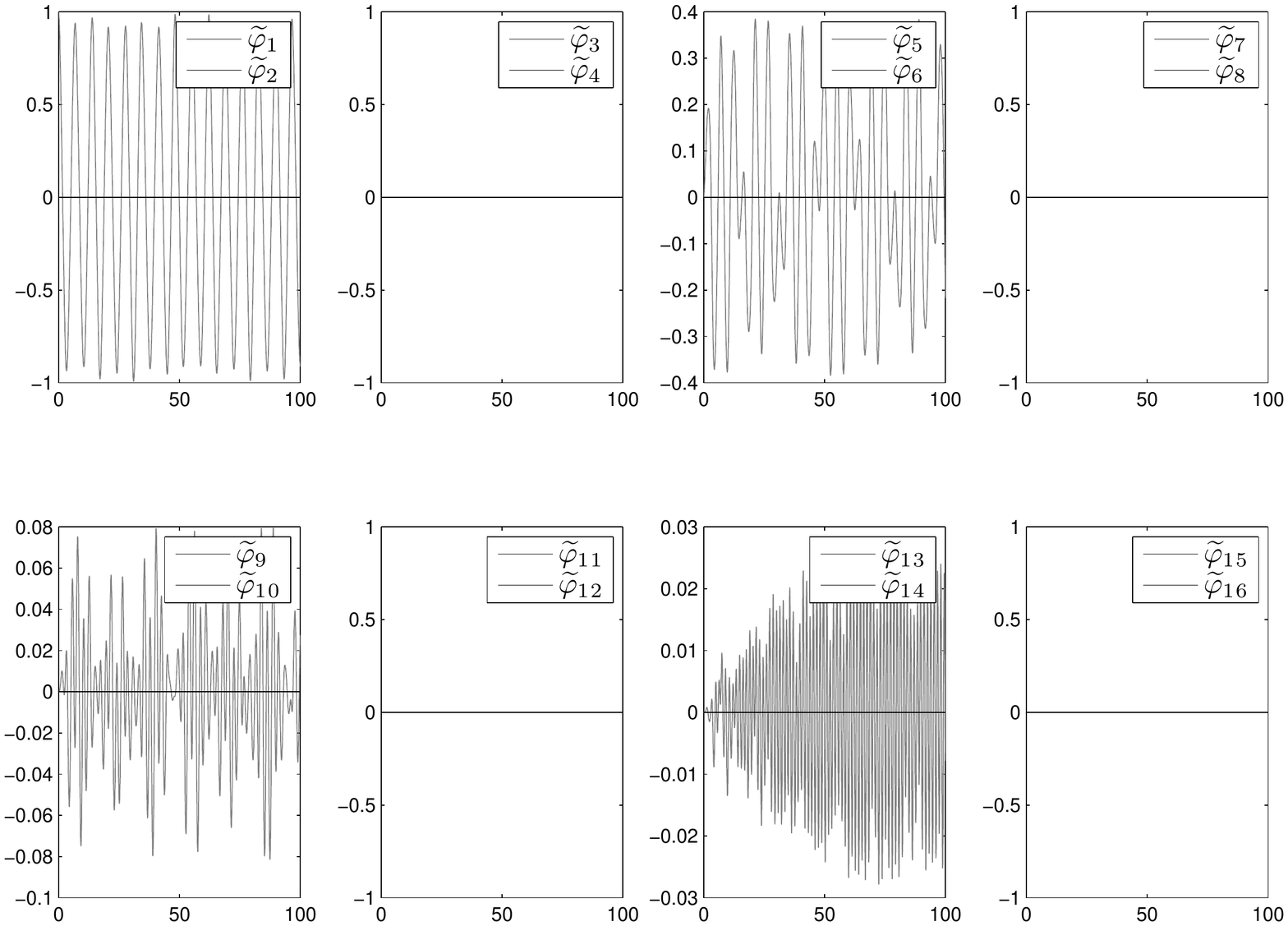}}
\caption{The solution of \eqref{eq:sistema-completo} with $\widetilde\phi_0^2=1$, $\widetilde\phi_1^2=0$, $(\widetilde\phi_0^k,\widetilde\phi_1^k)=(0,0)$ for any $k\in\{1,\dots,16\}$, $k\neq 2$.}\label{f:influence-2}
\end{center}
\end{figure}
One can observe that only the functions $\widetilde\varphi_2$,
$\widetilde\varphi_6$, $\widetilde\varphi_{10}$,
$\widetilde\varphi_{14}$ are not identically equal to zero.\par A
further consequence of our approach deserves to be emphasized. We
have truncated an infinite dimensional system up the least 16
modes: although this truncation is legitimate for energy reasons
(see \cite{bergaz}), it may lead to some small glitches. Since we
are considering a truncated system with fourteen longitudinal
modes and two torsional modes, it happens that if $k\in
\{5,\dots,14\}$ then the mode $\widetilde\varphi_k$ influences no
other modes, hence the solution of \eqref{eq:sistema-completo}
with $(\widetilde\phi_0^k,\widetilde\phi_1^k)\neq (0,0)$ and
$(\widetilde\phi_0^j,\widetilde\phi_1^j)=(0,0)$ for any $j\in
\{1,\dots,16\}$, $j\neq k$, satisfies
$\widetilde\varphi_k\not\equiv 0$ and $\widetilde\varphi_j\equiv
0$ for any $j\in \{1,\dots,16\}$, $j\neq k$. This is another
consequence of Proposition \ref{p:influence} (i) which can be
summarized in the following

\begin{proposition}\label{c:1} Let $k\in \{5,\dots, 14\}$. Suppose that $(\widetilde\phi_0^k,\widetilde\phi_1^k)\neq (0,0)$ and that $(\widetilde\phi_0^j,\widetilde\phi_1^j)=(0,0)$ for any $j\in \{1,\dots,16\}$, $j\neq k$. Then the corresponding solution of \eqref{eq:sistema-completo} satisfies
$$\widetilde \varphi_j(t)=0 \qquad \text{for any } t\in \R \quad \text{and} \quad j\in \{1,\dots,16\}, j\neq k \, .$$
\end{proposition}

\subsection{Comparison of the results for the coupled and uncoupled systems}

In this section we show that the responses of the truncated system \eqref{eq:sistema-completo} are in line with the theoretical and numerical responses obtained
from the Hill equation \eq{Hill} in Sections \ref{secsuffcond} and \ref{numres}.\par
We first provide the definition of stability of solutions of \eqref{eq:sistema-completo} satisfying $(\widetilde \phi_0^{15},\widetilde\phi_1^{15})=(\widetilde\phi_0^{16},\widetilde\phi_1^{16})=(0,0)$, namely solutions of
\begin{equation} \label{eq:P0}
\tag{$P_0$}
\begin{cases}
\ds{\widetilde\varphi_k''(t)+(\gamma\widetilde
\mu_k+A_k)\widetilde\varphi_k(t)+\sum_{\underset{j_1\ge j_2\ge
j_3}{j_1,j_2,j_3=1}}^{16} B_{j_1 j_2 j_3 k}\, \widetilde
\varphi_{j_1}(t) \widetilde \varphi_{j_2}(t) \widetilde
\varphi_{j_3}(t)=0 \, , } & \ \ k\in \{1,\dots,16\} \, , \\[25pt]
\widetilde \varphi_k(0)=\widetilde \phi_0^k \, , \quad \widetilde \varphi_k'(0)=\widetilde \phi_1^k \, , & \ \  k\in \{1,\dots,14\} \, , \\[10pt]
\widetilde \varphi_{15}(0)=\widetilde \varphi'_{15}(0)=\widetilde \varphi_{16}(0)=\widetilde \varphi'_{16}(0)=0 \, .
\end{cases}
\end{equation}
In the sequel a solution of \eqref{eq:P0} will be denoted by $\Phi_0=(\widetilde\varphi_{1},\dots,\widetilde\varphi_{14},0,0)$, see Proposition \ref{p:uncoupled}.
We also denote by $\Phi=(\widetilde \varphi_1,\dots,\widetilde\varphi_{16})$ a general solution of \eqref{eq:sistema-completo}.

\begin{definition} \label{d:stability} Let $\widetilde \phi_0^1,\dots,\widetilde \phi_0^{14}$ and $\widetilde \phi_1^1,\dots,\widetilde \phi_1^{14}$ be fixed
and let $\Phi_0$ be the corresponding solution of \eqref{eq:P0}.
\begin{itemize}
\item[(i)] We say that a solution $\Phi_0$ of \eqref{eq:P0} is stable with respect to the first torsional mode if for any $\e>0$ there exists $\delta>0$
such that if $\max\{|\widetilde \phi_0^{15}|,|\widetilde \phi_1^{15}| \}<\delta$ and $(\widetilde\phi_0^{16},\widetilde\phi_1^{16})=(0,0)$, then the solution $\Phi$ of \eqref{eq:sistema-completo} corresponding to $\widetilde \phi_0^1,\dots,\widetilde \phi_0^{16}$ and $\widetilde \phi_1^1,\dots,\widetilde \phi_1^{16}$ satisfies
\begin{equation*}
|\widetilde \varphi_{15}(t)|<\e \, , \quad |\widetilde \varphi_{15}'(t)|<\e   \qquad \text{for any } t\in \R \, .
\end{equation*}

\item[(ii)] We say that a solution $\Phi_0$ of \eqref{eq:P0} is stable with respect to the second torsional mode if for any $\e>0$ there exists $\delta>0$
such that if $\max\{|\widetilde \phi_0^{16}|,|\widetilde \phi_1^{16}| \}<\delta$ and $(\widetilde\phi_0^{15},\widetilde\phi_1^{15})=(0,0)$, then the solution
$\Phi$ of \eqref{eq:sistema-completo} corresponding to $\widetilde \phi_0^1,\dots,\widetilde \phi_0^{16}$ and $\widetilde \phi_1^1,\dots,\widetilde \phi_1^{16}$ satisfies
\begin{equation*}
|\widetilde \varphi_{16}(t)|<\e \, , \quad |\widetilde \varphi_{16}'(t)|<\e \qquad \text{for any } t\in \R \, .
\end{equation*}
\end{itemize}
\end{definition}

Definition \ref{d:stability} should be compared with Definition \ref{newdeff}. So far, to find sufficient conditions on $\Phi_0$ which guarantee its stability
with respect to one of the torsional modes (according to Definition \ref{d:stability}) seems out of reach. However, Theorems \ref{stable}
and \ref{stable2} suggest the following

\begin{conjecture} \label{conj:1} Let $l\in\{1,2\}$. There exists $A_l>0$ such that if
$$
\max\{|\widetilde\phi_0^1|,\dots,|\widetilde\phi_0^{14}|,|\widetilde\phi_1^1|,\dots,|\widetilde\phi_1^{14}|\}<A_l
$$
then the corresponding solution $\Phi_0$ of \eqref{eq:P0} is stable with respect to the $l$-th torsional mode.
\end{conjecture}

Our purpose is to support Conjecture \ref{conj:1} from a numerical point of view under suitable restrictions on the values of the initial conditions $\widetilde\phi_0^1,\dots,\widetilde\phi_0^{14},\widetilde\phi_1^{1},\dots,\widetilde\phi_1^{14}$ in \eqref{eq:P0}. For more numerical results concerning
Conjecture \ref{conj:1} we refer to \cite{bfg}.\par
For any $k\in \{1,\dots,14\}$ we consider the problem
\begin{equation} \label{eq:P0k}
\tag{$P_{0,k}$}
\begin{cases}
\ds{\widetilde\varphi_j''(t)+(\gamma\widetilde
\mu_j+A_j)\widetilde\varphi_j(t)+\sum_{\underset{j_1\ge j_2\ge
j_3}{j_1,j_2,j_3=1}}^{16} B_{j_1 j_2 j_3 j}\, \widetilde
\varphi_{j_1}(t) \widetilde \varphi_{j_2}(t) \widetilde
\varphi_{j_3}(t)=0 \, , } & \ \ j\in \{1,\dots,16\} \, , \\[25pt]
\widetilde \varphi_k(0)=A \, , \quad \widetilde \varphi_k'(0)=0  \, , \\[10pt]
\widetilde \varphi_{j}(0)=\widetilde \varphi'_{j}(0)=0\, ,  &  \ \ j\in \{1,\dots,16\} \,   , j\neq k \, .
\end{cases}
\end{equation}
Then for any $k\in \{1,\dots,14\}$ and $l\in \{15,16\}$ we also consider the following perturbation of \eqref{eq:P0k}
\begin{equation} \label{eq:P-eps-k}
\tag{$P_{\delta,k,l}$}
\begin{cases}
\ds{\widetilde\varphi_j''(t)+(\gamma\widetilde
\mu_j+A_j)\widetilde\varphi_j(t)+\sum_{\underset{j_1\ge j_2\ge
j_3}{j_1,j_2,j_3=1}}^{16} B_{j_1 j_2 j_3 j}\, \widetilde
\varphi_{j_1}(t) \widetilde \varphi_{j_2}(t) \widetilde
\varphi_{j_3}(t)=0 \, , } \qquad \quad  j\in \{1,\dots,16\} \, , \\[25pt]
\widetilde \varphi_k(0)=A \, , \quad \widetilde \varphi_k'(0)=0  \, , \\[10pt]
\widetilde \varphi_l(0)=\widetilde \varphi_l'(0)=\delta \, , \\[10pt]
\widetilde \varphi_{j}(0)=\widetilde \varphi'_{j}(0)=0\, , \qquad \qquad \qquad \qquad \qquad \qquad  j\in \{1,\dots,16\} \,   , j\neq k  \, , j\neq l \, .
\end{cases}
\end{equation}

In the particular case of solutions of \eqref{eq:P0k}, Conjecture \ref{conj:1} reads
\begin{center}
{\it for any $l\in\{1,2\}$ and $k\in\{1,\dots,14\}$ there exists $\widetilde A_l(k)>0$ such that for any $A\in(0,\widetilde A_l(k))$
the solution $\Phi_0$ of \eqref{eq:P0k} is stable with respect to the $l$-th torsional mode}\, .
\end{center}
Numerical simulations give the approximate values of $\widetilde A_l(k)$ as reported in Table \ref{wideA}.

\begin{table}[ht]
\begin{center}
{\small
\begin{tabular}{|c|c|c|c|c|c|c|c|c|c|c|c|c|c|c|}
\hline
$k$ & 1 & 2 & 3 & 4 & 5 & 6 & 7 & 8 & 9 & 10 & 11 & 12 & 13 & 14 \\
\hline
$\widetilde A_1(k)$ &  2.38 & 1.89 & 3.83 & 2.27 & 1.92 & 1.66 &  1.02  & $>10$ & $>10$ & $>10$ &  0.62  &  1.08  & 1.46   & 1.83 \\
\hline
$\widetilde A_2(k)$ &  5.17 & 4.38 & 4.94 & 8.08 & 4.18 & 4.05 &  3.87  &  3.59  &  3.10 &  1.87 & $>10$ & $>10$ & $>10$ & $>10$ \\
\hline
\end{tabular}}
\caption{Numerical values of $\widetilde A_1(k)$ and $\widetilde A_2(k)$ when $\gamma=5.17\cdot 10^{-4}$.}\label{wideA}
\end{center}
\end{table}

\begin{remark}\label{secondarem}
{\rm From Table \ref{wideA} we see that $\widetilde A_1(k)>\widetilde A_2(k)$ provided that $k=8,9,10$ whereas for lower $k$ we have
that $\widetilde A_1(k)<\widetilde A_2(k)$.}
\end{remark}

Let us explain how we obtained Table \ref{wideA}. As a particular example we take the value of $A_1(5)$. In Figure \ref{f:influence-5} we plot the
graph of the component $\widetilde\varphi_5$ of the solution $\Phi_0$ of \eqref{eq:P0k} with $k=5$ and $A=1.92$.

\begin{figure}[ht]
\begin{center}
{\includegraphics[height=40mm, width=100mm]{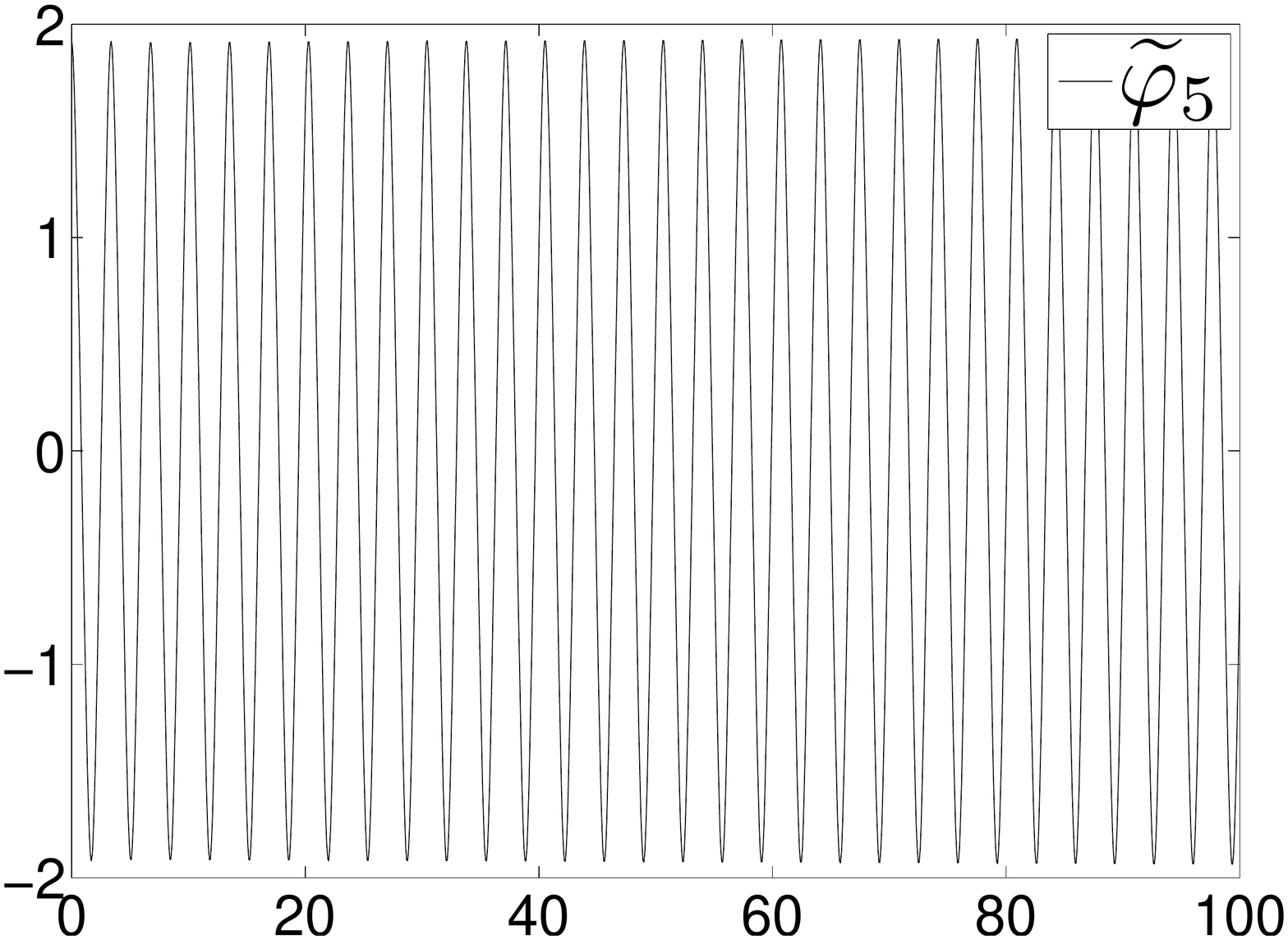}}
\caption{The component $\widetilde\varphi_5$ of the solution $\Phi_0$ of \eqref{eq:P0k} with $k=5$ and $A=1.92$.}\label{f:influence-5}
\end{center}
\end{figure}

The instability phenomenon appearing for $A=1.92$ is represented in Figure \ref{f:influence-5-pert} where we see that the amplitude of the
oscillations of $\widetilde\varphi_{15}$ is fairly wide for large $t$ even if we choose a small value of $\delta$ in \eqref{eq:P-eps-k}.

\begin{figure}[ht]
\begin{center}
{\includegraphics[height=90mm, width=170mm]{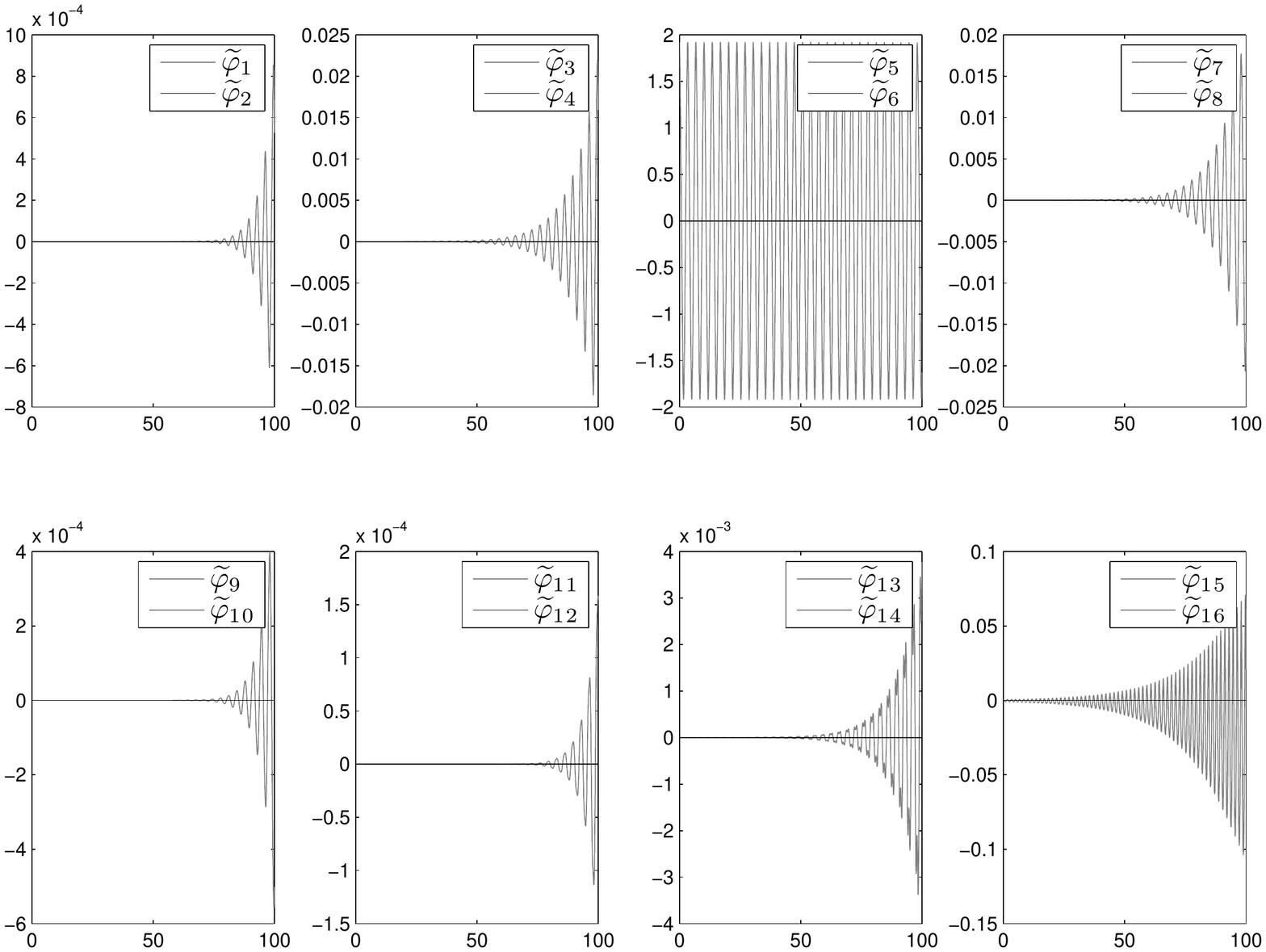}}
\caption{The components of the solution $\Phi$ of \eqref{eq:P-eps-k} with $k=5$, $l=15$, $A=1.92$ and $\delta=5\cdot 10^{-4}$.}\label{f:influence-5-pert}
\end{center}
\end{figure}

On the other hand, if we choose $A=1.91$ we see in Figure
\ref{f:1.91} that the amplitude of the oscillations of
$\widetilde\varphi_{15}$ remains of the same order of magnitude
also when $t$ is large. This is how we obtained $A_1(5)=1.92$ in
Table \ref{wideA}.

\begin{figure}[ht]
\begin{center}
{\includegraphics[height=50mm, width=90mm]{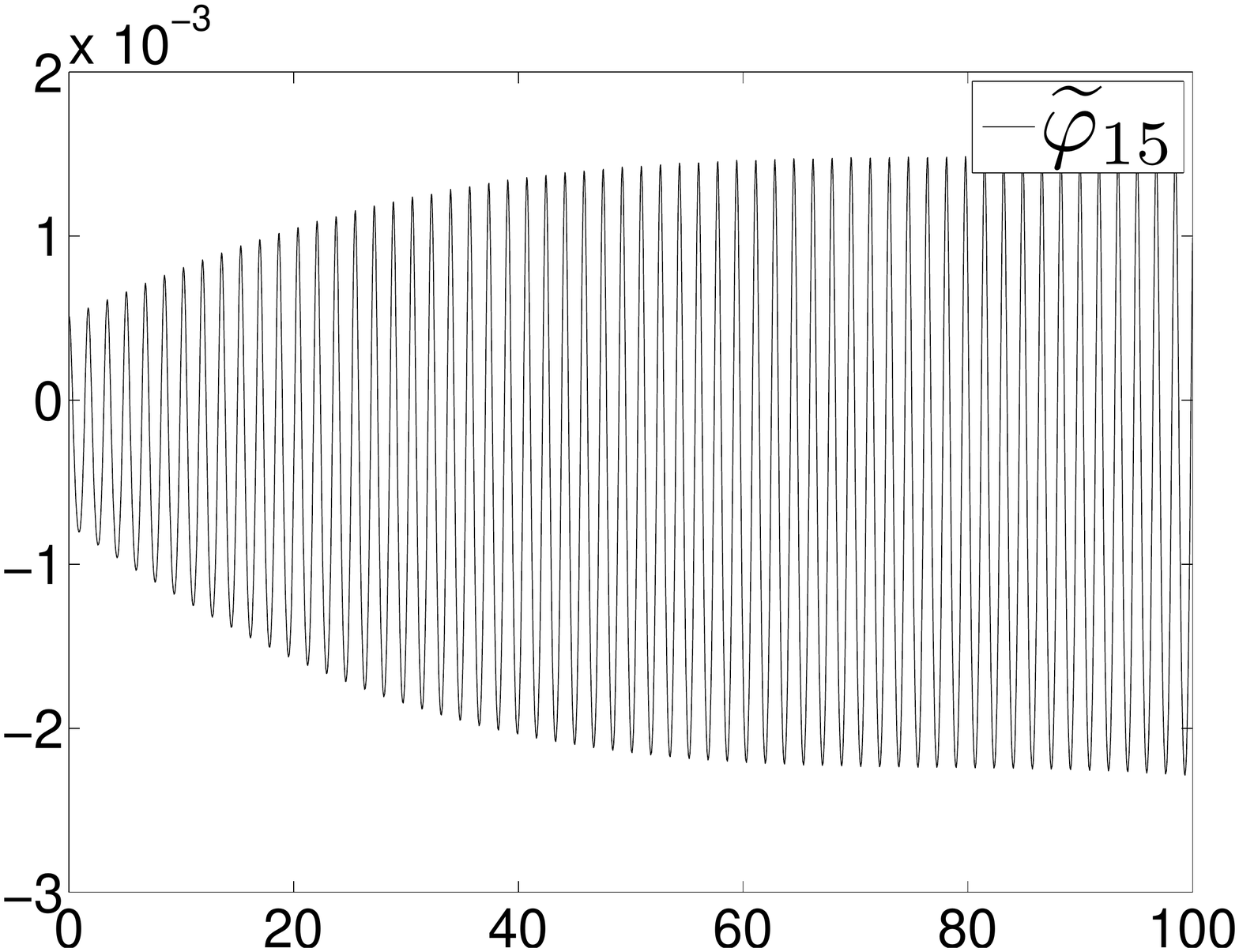}}
\caption{The function $\widetilde\varphi_{15}$ corresponding to the solution of \eqref{eq:P-eps-k} with $k=5$, $l=15$, $A=1.91$ and $\delta=5\cdot 10^{-4}$.} \label{f:1.91}
\end{center}
\end{figure}

The comparison between Table \ref{wideA} and Table
\ref{threegamma} (middle table for $\gamma=5.17\cdot 10^{-4}$)
deserves several comments. The values of $\widetilde A_1(k)$ when
$k\in \{5,\dots,14\}\setminus\{8,9,10\}$ and of $\widetilde
A_2(k)$ when $k\in \{5,\dots,10\}$ are very close to the
corresponding ones of $A_1(k)$ and $A_2(k)$ in Table
\ref{threegamma}. The reason is that for these values of $k$ the
solution $\Phi_0$ of \eqref{eq:P0k} satisfies
$\widetilde\varphi_j\equiv 0$ for any $j\neq k$, see Proposition
\ref{c:1}, hence the equations are decoupled. \par Then, with the
same choice of $k$, fix $\delta>0$, $l\in \{15,16\}$ and let
$\Phi_\delta$ be the corresponding solution of \eqref{eq:P-eps-k}.
If $\delta$ is small enough the function $\Phi_\delta$ is expected
to be close to $\Phi_0$ until $\widetilde\varphi_l$ remains small
together with its derivative. And indeed, if we chose $k=5$,
$l=15$, $A=1.92$, $\delta=5\cdot 10^{-4}$, we numerically obtain
the plots of Figure \ref{f:influence-5-pert}.\par We also observe
that for larger values of $A$, say $A=1.94$, when the fifteenth
component of $\Phi_\delta$ is no longer ``negligible'' the values
of the fifth component $\widetilde\varphi_5$ of $\Phi_\delta$ are
considerably different from the values of the fifth component of
$\Phi_0$, at least for large $t$: in the latter case
$\widetilde\varphi_5$ is periodic while in the former case
$\widetilde\varphi_5$ exhibits a variation in the amplitude of the
oscillations soon after the amplitude of the oscillations of the
fifteenth component has become relatively large, as one can see
from Figure \ref{f:1.94}.

\begin{figure}[ht]
\begin{center}
{\includegraphics[height=50mm, width=120mm]{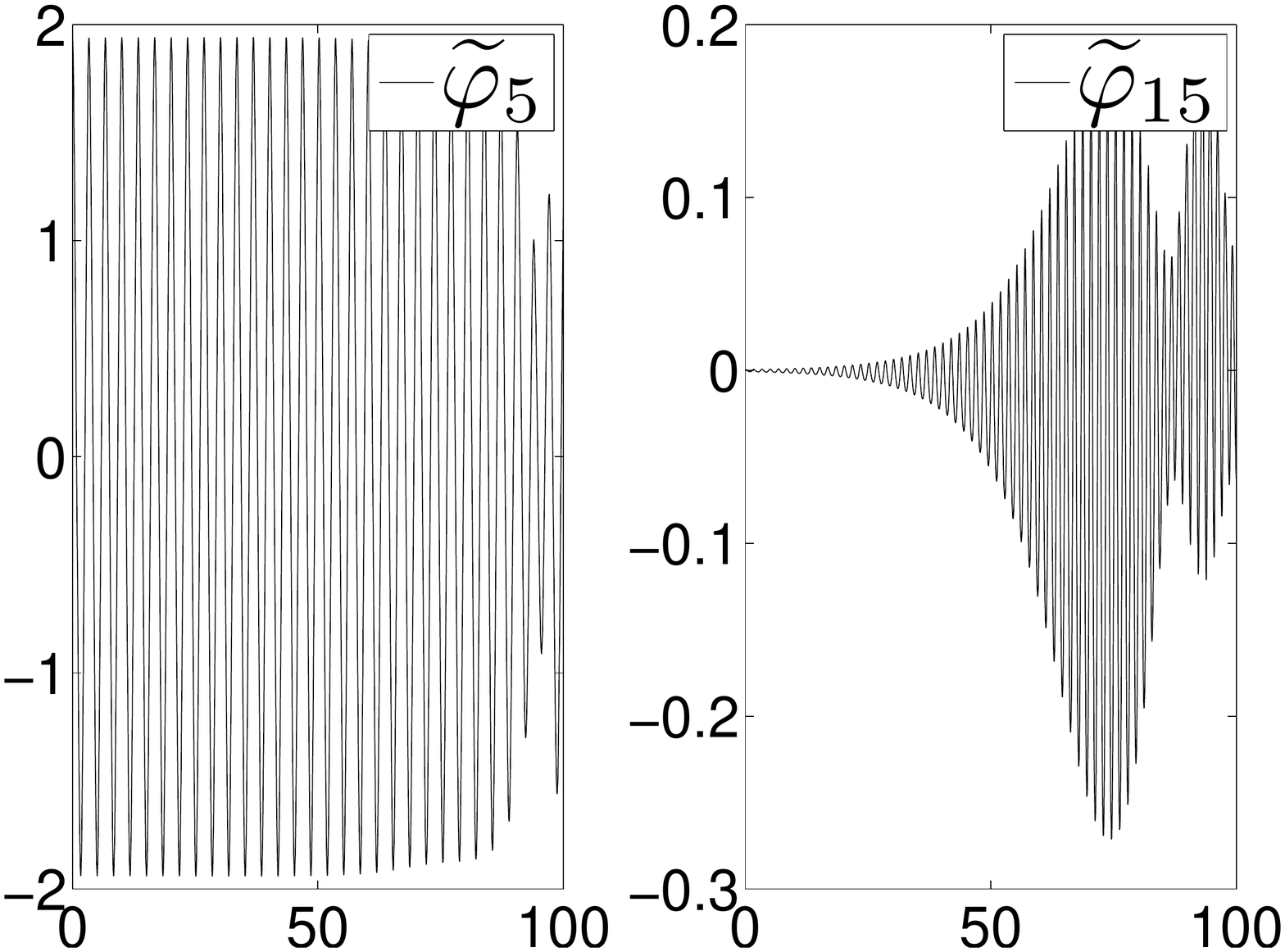}}
\caption{The functions $\widetilde\varphi_5$ and $\widetilde\varphi_{15}$ corresponding to the solution $\Phi$ of \eqref{eq:P0k} with $k=5$ and
$A=1.94$.}\label{f:1.94}
\end{center}
\end{figure}

The critical values $\widetilde A_1(8)$, $\widetilde A_1(9)$,
$\widetilde A_1(10)$, $\widetilde A_2(11)$, $\widetilde A_2(12)$,
$\widetilde A_2(13)$, $\widetilde A_2(14)$ are very large in both
Tables \ref{threegamma} and \ref{wideA}. No evident instability
behavior was detected when considering solutions of \eqref{ode1}
and \eqref{Hill} on one hand and solutions of \eqref{eq:P-eps-k}
on the other hand.\par It remains to discuss the cases where $k\in
\{1,2,3,4\}$ for which the values of $\widetilde A_1(k)$ and
$\widetilde A_2(k)$ in Table \ref{wideA} are quite different from
the values of $A_1(k)$ and $A_2(k)$ in the middle Table
\ref{threegamma}. By looking at Figures \ref{f:influence-1} and
\ref{f:influence-2} one can see the graphs of the components of
$\Phi_0$, the solution of \eqref{eq:P0k}, respectively in the
cases $k=1$ and $k=2$ (for the two remaining cases $k=3,4$ we
refer to \cite{bfg}). In the two figures one observes that the
$k$-th component of $\Phi_0$ is not the only large one, there are
several other large components. In this case, the reduction of
\eqref{eq:sistema-completo} to equations \eqref{ode1} and
\eqref{Hill} {\em is not a good approximation of the real
situation}.

%For instance, take $k=1$, $l=15$, $A=3.02$ and $\delta=5\cdot
%10^{-4}$ in \eqref{eq:P-eps-k}; Figure \ref{f:instability-1-15}
%shows instability even if $A=3.02$ is smaller than the critical
%value $A_1(1)=3.32$ found in Table \ref{threegamma}.

%\begin{figure}[ht]
%\begin{center}
%{\includegraphics[height=40mm, width=100mm]{varphi_1-15.pdf}}
%\caption{The function $\widetilde\varphi_{15}$ corresponding to the solution of \eqref{eq:P-eps-k} with $k=1$, $l=15$, $A=3.02$ and $\delta=5\cdot 10^{-4}$.}\label{f:instability-1-15}
%\end{center}
%\end{figure}

\subsection{Thresholds of stability in the TNB}  \label{s:TNB}
In this subsection, starting from the values obtained in Table
\ref{wideA}, we provide estimates on the thresholds of stability
for the TNB by using the isolated system \eqref{wave-3} with the
values of the parameters given in Section \ref{numres}. According
to the notations of Subsection \ref{truncsys}, we denote by
$\widetilde w_k$, $k\in \{0,\dots,14\}$ the first fourteen
longitudinal eigenfunctions normalized in $L^2(\Omega)$. For any
$k\in \{0,\dots,14\}$ and $l\in\{1,2\}$ we choose the initial
conditions in \eqref{wave-3} in the form
\begin{equation} \label{eq:u000}
u_0(x,y)=\widetilde A_l(k) \widetilde w_k(x,y) \, , \quad
u_1(x,y)=0 \qquad \text{for any } (x,y)\in (0,\pi)\times
(-\ell,\ell)
\end{equation}
where the $\widetilde A_l(k)$ are as in Table \ref{wideA}. Our
purpose is to provide, for any $k\in \{1,\dots,14\}$, the values
measured in meters of the stability threshold for the energy
transfer from the $k$-th longitudinal mode to a torsional mode.
This will give an idea of the initial displacement of the deck
sufficient to activate the torsional oscillations. To this end, we
first give in Table \ref{t:L-infty-norms} the approximate values
of the $L^\infty$-norms of the $L^2$-normalized eigenfunctions.
\begin{table}[ht]
\begin{center}
{\small
\begin{tabular}{|c|c|c|c|c|c|c|c|c|c|c|c|c|c|c|}
\hline
$k$ & 1 & 2 & 3 & 4 & 5 & 6 & 7 & 8 & 9 & 10 & 11 & 12 & 13 & 14 \\
\hline
$\|\widetilde w_k\|_{L^\infty}$ \! & \!3.897\! & \!3.899\! & \!3.899\! & \!3.900\! & \!3.901\! & \!3.902\! & \!3.904\! & \!3.905\! & \!3.907\! & \!3.909\! & \!3.912\! & \!3.914\! & \!3.917\! & \!3.920 \\
\hline
\end{tabular}}
\caption{Approximate values of $\|\widetilde w_k\|_{L^\infty}$ for
$k\in \{1,\dots,14\}$.} \label{t:L-infty-norms}
\end{center}
\end{table}

By \eqref{eq:u000}, Table \ref{wideA}, Table \ref{t:L-infty-norms}
and \eqref{gamma} with $k_1=k_2=\frac{6000H}{L^3}$ we obtain the
following table which shows that the instability amplitude has the
same order of magnitude as observed for the TNB, see
\cite{tacoma}.

\begin{table}[ht]
\begin{center}
{\small
\begin{tabular}{|c|c|c|c|c|c|c|c|c|c|c|c|c|c|c|}
\hline
$k$ & 1 & 2 & 3 & 4 & 5 & 6 & 7 & 8 & 9 & 10 & 11 & 12 & 13 & 14 \\
\hline
$\|u_0\|_{L^\infty}$ & \!9.27\! & \!7.37\! & \!14.93\! & \!8.85\! & \!7.45\! & \!6.48\! & \!3.98\! & \!$>40$\! & \!$>40$\! & \!$>40$\! & \!2.43\! & \!4.23\! & \!5.72\! & \!7.17 \\
\hline
$\|u_0\|_{L^\infty}$ & \!20.15\! & \!17.08\! & \!19.26\! & \!31.51\! & \!16.30\! & \!15.80\! & \!15.11\! & \!14.02\! & \!12.11\! & \!7.31\! & \!$>40$\! & \!$>40$\! & \!$>40$\! & \!$>40$ \\
\hline
\end{tabular}}
\caption{The $L^\infty$-norm of $u_0$ measured in meters
corresponding to the stability threshold of the $k$-th
longitudinal mode with respect to the first torsional mode (first
line) and the second torsional mode (second line).}
\label{t:displacement}
\end{center}
\end{table}

\section{Proof of Theorem \ref{stable}}\label{stable proof}

Assume that $1\leq k\leq 14$ and $l\in\{1,2\}$ are given. We introduce two constants which will be useful in the sequel. We set
\neweq{mlk}
\delta_l:=\gamma \nu_{l,2}+\bar a_l\,,\qquad \rho_k:=\gamma \mu_{k,1}+a_k\,.
\endeq
Using the definition of $\rho_k$ in \eq{mlk}, for any $E>0$ we put
\neweq{Lambdapm}
\Lambda_{\pm}^k(E)=\frac{\sqrt{\rho_k^2+4b_kE}\pm\rho_k}{b_k}\,.
\endeq
Then \eq{energyEk} reads
\neweq{quadrato}
(\overline\varphi'_k)^2=\frac{b_k}{2}(\Lambda_{+}^k(E)+\overline\varphi_k^2)(\Lambda_{-}^k(E)-\overline\varphi_k^2)\,.
\endeq
Hence, since any $k$-th longitudinal mode $\overline \varphi_k$ satisfies \eq{energyEk}, we deduce
\begin{equation}\label{gksupn}
\|\overline \varphi_k\|_\infty=\sqrt{{\Lambda_{-}^k(E)}}\,.
\end{equation}

Since in the equation \eq{ode1} there are only odd terms, the period $T_k(E)$ of $\overline\varphi_k$ is the double of the width of an interval of
monotonicity of $\overline{\varphi}_k$. As the problem is autonomous, we may assume that $\overline{\varphi}_k(0)=-\|\overline{\varphi}_k\|_\infty$
and $\overline{\varphi}_k'(0)=0$. By symmetry and periodicity we then have that $\overline{\varphi}_k(T_k/2)=\|\overline{\varphi}_k\|_\infty$ and
$\overline{\varphi}_k'(T_k/2)=0$. In the first interval of monotonicity of $\overline{\varphi}_k$ we may take the square root of \eq{quadrato} and obtain
\neweq{squareroot}
\overline\varphi'_k(t)=\sqrt{\frac{b_k}{2}(\Lambda_{+}^k(E)+\overline\varphi_k^2(t))(\Lambda_{-}^k(E)-\overline\varphi_k^2(t))}\qquad\forall t\in\left[0,\frac{T_k}{2}\right]\,.
\endeq
By separating variables, and integrating over the interval $(0,T_k/2)$ we obtain
$$\frac{T_k(E)}{2}=\sqrt{\frac{2}{b_k}}\int_{-\|\overline{\varphi}_k\|_\infty}^{\|\overline{\varphi}_k\|_\infty}
\frac{ds}{\sqrt{(\Lambda_+^k(E)+s^2)(\Lambda_-^k(E)-s^2)}}\,.$$
Using the fact that the integrand is even with respect to $s$ and with a change of variable we get
\neweq{period}
T_k(E)=\frac{4\sqrt{2}}{\sqrt{b_k}}\int_0^1\frac{ds}{\sqrt{(\Lambda_+^k(E)+\Lambda_-^k(E)s^2)(1-s^2)}}\, .
\endeq
In particular,
\neweq{map}
\mbox{the map }E\mapsto T_k(E)\mbox{ is strictly decreasing and }\lim_{E\to0}T_k(E)=T_k(0)=\frac{2\pi}{\sqrt{\rho_k}}\, .
\endeq

Let us prove the following sufficient condition for the torsional stability of $\overline \varphi_k$.

\begin{lemma}\label{suffcond}
If there exists an integer $m$ such that
\neweq{zukn}
\frac{4m^2\pi^2}{\delta_l}\le T_k(E)^2 \le \frac{4(m+1)^2\pi^2}{\delta_l+d_{l,k}\Lambda_-^k(E)}
\endeq
then $\overline \varphi_k$ is stable with respect to the $l$-th torsional mode.
\end{lemma}
\begin{proof} With the initial conditions $\xi(0)=\dot{\xi}(0)=0$, the unique solution of \eq{Hill} is $\xi\equiv0$. The statement follows if we
prove that \eq{zukn} is a sufficient condition for the trivial solution $\xi\equiv0$ to be stable in the Lyapunov sense, namely if the solutions of
\eq{Hill} with small initial data $|\xi(0)|$ and $|\dot{\xi}(0)|$ remain small for all $t\ge0$.\par
Since $\overline \varphi_k$ is $T_k$-periodic, the function $\overline \varphi_k^2$ is $T_k/2$-periodic. Then $A_{l,k}(t)$ is a positive $T_k/2$-periodic
function and a stability criterion for the Hill equation due to Zhukovskii \cite{zhk}, see also \cite[Chapter VIII]{yakubovich}, states that the trivial
solution of \eq{Hill} is stable provided that
$$
\frac{4 m^2\pi^2}{T_k(E)^2}\leq A_{l,k}(t)\leq \frac{4 (m+1)^2\pi^2}{T_k(E)^2}\qquad\forall t\in\R\,
$$
for some integer $m\geq 0$. By recalling \eq{mlk} and the definition of $A_{l,k}$ in \eq{Alk}, this condition is equivalent to
$$
\frac{4 m^2\pi^2}{T_k(E)^2}\leq \delta_l+ d_{l,k}\overline \varphi_k^2(t)\leq \frac{4 (m+1)^2\pi^2}{T_k(E)^2}\qquad\forall t\in\R\, .
$$
In turn, by invoking \eq{gksupn} we see that the latter is equivalent to \eq{zukn}.\end{proof}

\begin{remark}{\em For $E=0$ the right hand side of \eq{zukn} equals $\frac{4(m+1)^2\pi^2}{\delta_l}$ which is strictly larger than the left hand side
of \eq{zukn}. But the right hand side of \eq{zukn} is strictly decreasing with respect to $E$ and tends to $0$ as $E\to\infty$. Therefore, the
interval defined by \eq{zukn} is empty for sufficiently large $E$.}
\end{remark}

Let
\neweq{mm}
m=\max\left\{k\in\N;\, k<\sqrt{\frac{\delta_l}{\rho_k}}\right\}
\endeq
so that, by \eq{map}, $\frac{4m^2\pi^2}{\delta_l}<T_k(0)^2$. We then infer that there exists $E_1(l,k)>0$ such that
\neweq{primastima}
\frac{4m^2\pi^2}{\delta_l}\le T_k(E)^2\qquad\forall E\le E_1(l,k)\, .
\endeq
This gives a sufficient condition for the first inequality in \eq{zukn} to be satisfied.\par
In view of \eq{mm} and of the assumption \eq{strange} we know that
$$T_k(0)^2=\frac{4\pi^2}{\rho_k}<\frac{4(m+1)^2\pi^2}{\delta_l}=\frac{4(m+1)^2\pi^2}{\delta_l+d_{l,k}\Lambda_-^k(0)}\, .$$
Then the continuity of the maps $E\mapsto T_k(E)$ and $E\mapsto \Lambda_-^k(E)$ implies that there exists $E_2(l,k)>0$ such that
\neweq{secondastima}
T_k(E)^2\le \frac{4(m+1)^2\pi^2}{\delta_l+d_{l,k}\Lambda_-^k(E)}\qquad\forall E\le E_2(l,k)\, .
\endeq
This gives a sufficient condition for the second inequality in
\eq{zukn} to be satisfied.\par We may now conclude the proof of
Theorem \ref{stable}. We choose $m$ as in \eq{mm}, we define
$E_1(l,k)$ and $E_2(l,k)$ as in
\eqref{primastima}-\eqref{secondastima}, and we take
$$E\le E_k^l:=\min\{E_1(l,k),E_2(l,k)\}$$
so that both \eq{primastima} and \eq{secondastima} are satisfied. Then Lemma \ref{suffcond} tells us that $\overline \varphi_k$ is stable with respect
to the $l$-th torsional mode. By taking $\Lambda=\Lambda_-^k$ as in \eq{Lambdapm} we readily obtain the $L^\infty$-bound for $\overline \varphi_k$
and the proof of Theorem \ref{stable} is so completed.\par\medskip

The just completed proof also enables us to give the following quantitative version of Theorem \ref{stable}.

\begin{theorem}\label{stableprecise}
Fix $1\leq k\leq 14$, $l\in\{1,2\}$ and let $E_k^l:=\min\{E_1(l,k),E_2(l,k)\}$, see \eqref{primastima}-\eqref{secondastima}.
Then the $k$-th longitudinal mode $\overline \varphi_k$ at energy $E(\phi_0^k,\phi_1^k)$ is stable with respect to the $l$-th torsional mode provided that
$E\le E_k^l$ or, equivalently, provided that $\|\overline \varphi_k\|_\infty^2\le \Lambda_-^k(E_k^l)$ where $\Lambda_-^k$ is defined in \eqref{Lambdapm}.
\end{theorem}

\section{Proof of Theorem \ref{stable2}}\label{secondproof}

The proof of Theorem \ref{stable2} follows the same lines of the proof of Theorem \ref{stable}, see Section \ref{stable proof}.
But the continuity of the maps involved is not enough to obtain the desired inequality and a different stability criterion is needed.\par
Note that by \eq{mlk} the condition \eq{strange2} reads
\neweq{stranger}
\frac{\delta_l}{\rho_k}=(m+1)^2\, .
\endeq

The right hand side of \eq{period} is an elliptic integral of the first kind: after the further change of variables $s=\cos\phi$ it may be written as
\neweq{elliptic}
T_k(E)=\frac{4\sqrt{2}}{\sqrt{b_k}}\int_0^{\pi/2}\frac{d\phi}{\sqrt{\Lambda_+^k(E)+\Lambda_-^k(E)\cos^2\phi}}=
\frac{4}{\sqrt[4]{\rho_k^2+4b_k E}}\int_0^{\pi/2}\frac{d\phi}{\sqrt{1-\mu_k(E)\sin^2\phi}}
\endeq
where
\neweq{muk}
\mu_k(E)=\frac12 \left(1-\frac{\rho_k}{\sqrt{\rho_k^2+4b_k E}}\right)\ \in\ \left(0,\frac12\right)\, .
\endeq
This enables us to compute the derivative of $T_k(E)$ for $E=0$.

\begin{lemma}\label{derivataT}
Let $T_k=T_k(E)$ be the function in \eqref{elliptic}. Then
$$T_k'(0)=-\frac{3\, \pi\, b_k}{2\, \rho_k^{5/2}}\, .$$
\end{lemma}
\begin{proof} An asymptotic expansion of $\mu_k$ in \eq{muk} shows that
$$\mu_k(E)\sim \frac{b_k}{\rho_k^2}\, E\qquad\mbox{as }E\to0\, .$$
Then we may also expand $T_k(E)$ in \eq{elliptic} as follows
\begin{eqnarray*}
T_k(E) &\sim& \frac{4}{\sqrt{\rho_k}}\left(1-\frac{b_k}{\rho_k^2}\, E\right)\int_0^{\pi/2}\frac{d\phi}{\sqrt{1-\frac{b_k}{\rho_k^2}E\sin^2\phi}}\\
\ &\sim& \frac{4}{\sqrt{\rho_k}}\left(1-\frac{b_k}{\rho_k^2}\, E\right)\int_0^{\pi/2}\left(1+\frac{b_k}{2\rho_k^2}E\sin^2\phi\right)\, d\phi\\
\ &=& \frac{2\pi}{\sqrt{\rho_k}}\left(1-\frac{b_k}{\rho_k^2}\, E\right)\left(1+\frac{b_k}{4\rho_k^2}\, E\right)\quad\mbox{as }E\to0\, .
\end{eqnarray*}
This proves that
$$T_k(E)=T_k(0)-\frac{3\, \pi\, b_k}{2\, \rho_k^{5/2}}\, E+o(E)\qquad\mbox{as }E\to0$$
and the statement follows.\end{proof}

Due to the presence of odd terms, the solution of \eq{ode1} has several symmetry properties that we summarize in the next statement.
We also provide a pointwise upper bound for $\overline \varphi_k$.

\begin{lemma}\label{pointbound}
Let $\psi$ be the unique (periodic) solution of the problem
\neweq{psiorigin}
\begin{cases}
\psi''(t)+\rho_k\psi(t)+b_k \psi^3(t)=0 \qquad\forall t>0 \\
\psi(0)=0\, ,\quad \psi'(0)=\sqrt{2E}
\end{cases}
\endeq
and denote by $T$ its period. Then:\par
(i) $\psi(T/4)=\max_t \psi(t)$, $\psi(T/4-t)=\psi(T/4+t)$ for all $t$, $\psi(T/2+t)=-\psi(t)$ for all $t$;\par
(ii) the following estimate holds
$$0\le\psi(t)\le\min\left\{\sqrt{\frac{2E}{\rho_k}}\, \sin\left(\sqrt{\rho_k}\, t\right),\sqrt{{\Lambda_{-}^k(E)}}\right\}
\qquad\forall\ 0\le t\le \frac{T}4 \, .$$
\end{lemma}
\begin{proof} The symmetry properties in $(i)$ are well-known calculus properties.\par
In order to prove $(ii)$, we observe that by conservation of energy and arguing as for \eq{squareroot}, we find
$$
\psi'=\sqrt{2E-\rho_k\psi^2-\frac{b_k}{2}\psi^4}\qquad\forall 0\le t\le \frac{T}4\,.
$$
By separating variables we then obtain
$$
\frac{d\psi}{\sqrt{2E-\rho_k\psi^2-\frac{b_k}{2}\psi^4}}=dt\, ;
$$
in turn, by integrating over $[0,t]$ and dropping the fourth power term, we get
$$
\frac1{\sqrt{\rho_k}}\arcsin\left(\sqrt{\frac{\rho_k}{2E}}\, \psi(t)\right)=\frac{1}{\sqrt{2E}}\int_0^{\psi(t)}\frac{d\xi}{\sqrt{1-\frac{\rho_k}{2E}\xi^2}}
\le\int_0^{\psi(t)}\frac{d\xi}{\sqrt{2E-\rho_k\xi^2-\frac{b_k}{2}\xi^4}}=\int_0^td\tau=t\, .
$$
This yields
$$\psi(t)\le \sqrt{\frac{2E}{\rho_k}}\, \sin\left(\sqrt{\rho_k}\, t\right)\qquad\forall\ 0\le t\le \frac{T}4$$
while from \eq{gksupn} we know that $\psi(t)\le \sqrt{{\Lambda_{-}^k(E)}}$ for all $t$. By taking the minimum, we infer the desired upper bound in $(ii)$.
\end{proof}

The previous lemmas enable us to prove the following technical result.

\begin{lemma}\label{suffcond2}
Assume that \eqref{strange3} and \eqref{stranger} hold. There
exists $E_2(l,k)>0$ such that if $E\le E_2(l,k)$ then
$$\int_0^{T_k(E)/2}\sqrt{A_{l,k}(t)}\, dt+\frac12 \log\left(\frac{\max_{t}\ A_{l,k}(t)}{\min_{t}\ A_{l,k}(t)}\right)\le (m+1)\pi\, .$$
\end{lemma}
\begin{proof} Recall that $A_{l,k}$ is a $\frac{T_k(E)}{2}$ --periodic function. Up to a time translation, we may assume that $\overline \varphi_k$
solves \eq{psiorigin}; then we estimate
\begin{eqnarray}
\int_0^{T_k(E)/2}\sqrt{A_{l,k}(t)}\, dt &=& 2\int_0^{T_k(E)/4}\sqrt{A_{l,k}(t)}\, dt=2\int_0^{T_k(E)/4}\sqrt{\delta_l
+d_{l,k}\overline \varphi_k^2(t)}\ dt \notag \\
\mbox{as }E\to0\quad &\sim& 2\sqrt{\delta_l}\int_0^{T_k(E)/4}\left(1+\frac{d_{l,k}}{2\delta_l}\, \overline \varphi_k^2(t)\right)\, dt \notag \\
\mbox{by Lemma \ref{pointbound}}\quad &\le& \frac{\sqrt{\delta_l}}{2}\, T_k(E)
+\frac{2\, d_{l,k}\, E}{\rho_k\, \sqrt{\delta_l}}\int_0^{T_k(E)/4}\sin^2(\sqrt{\rho_k}\, t)\, dt \notag \\
\mbox{by Lemma \ref{derivataT}}\quad &\sim&
\frac{\sqrt{\delta_l}}{2}\left(T_k(0)-\frac{3\, \pi\, b_k}{2\,
\rho_k^{5/2}}\, E\right)+ \frac{2\, d_{l,k}\, E}{\rho_k\,
\sqrt{\delta_l}}\left(\frac{T_k(E)}{8}-\frac{\sin(\sqrt{\rho_k}\,
T_k(E)/2)}{4\, \sqrt{\rho_k}}\right) \label{stimona}
\end{eqnarray}

By Lemma \ref{derivataT} we also infer that
\begin{eqnarray*}
\sin\left(\frac{\sqrt{\rho_k}\, T_k(E)}{2}\right) &=&
\sin\left(\frac{\sqrt{\rho_k}\, T_k(0)}{2}-\frac{3\, \pi\, b_k}{4\, \rho_k^2}\, E+o(E)\right)\\
\mbox{by \eq{map}}\quad &=& \sin\left(\pi-\frac{3\, \pi\, b_k}{4\, \rho_k^2}\, E+o(E)\right)=o(1)\quad\mbox{as }E\to0\, .
\end{eqnarray*}
By plugging this information into \eq{stimona} we end up with
\begin{eqnarray}
\int_0^{T_k(E)/2}\sqrt{A_{l,k}(t)}\, dt &\le & \frac{\sqrt{\delta_l}}{2}\left(T_k(0)-\frac{3\, \pi\, b_k}{2\, \rho_k^{5/2}}\, E\right)+
\frac{d_{l,k}\, T_k(0)}{4\, \rho_k\, \sqrt{\delta_l}}\, E+o(E) \notag \\
\mbox{by \eq{stranger}}\quad &\sim & (m+1)\pi+\frac{\pi}{2\rho_k^2}\left(\frac{d_{l,k}}{m+1}-\frac{3(m+1)b_k}{2}\right)\, E\quad\mbox{as }E\to0\, . \label{endup}
\end{eqnarray}

On the other hand, by \eq{gksupn}, we know that the periodic coefficient $A_{l,k}$ in \eq{Hill} satisfies the following sharp bounds
$$\delta_l\le A_{l,k}(t)\le\delta_l+ d_{l,k}\Lambda_{-}^k(E)\qquad\forall t\ge0$$
where $\delta_l$ is defined in \eq{mlk}. In fact, $A_{l,k}$ has only two critical points in $[0,T_k(E)/2)$ and
$$\min_{t\in[0,T_k(E)/2)}\ A_{l,k}(t)=\delta_l\ ,\quad\max_{t\in[0,T_k(E)/2)}\ A_{l,k}(t)=\delta_l+ d_{l,k}\Lambda_{-}^k(E)\ .$$
In particular this means that
$$\log\left(\frac{\max_{t}\ A_{l,k}(t)}{\min_{t}\ A_{l,k}(t)}\right)=\log\left(1+\frac{d_{l,k}}{\delta_l}\Lambda_{-}^k(E)\right)$$
and, by \eq{stranger} and by taking advantage of the explicit expression in \eq{Lambdapm}, we obtain
\neweq{stimalog}
\log\left(\frac{\max_{t}\ A_{l,k}(t)}{\min_{t}\ A_{l,k}(t)}\right)\sim\frac{d_{l,k}}{\delta_l}\Lambda_{-}^k(E)\sim2\, \frac{d_{l,k}}{\delta_l\rho_k}\, E
=\frac{2\, d_{l,k}}{(m+1)^2\rho_k^2}\, E  \qquad\mbox{as }E\to0\, .
\endeq

By combining \eq{endup} with \eq{stimalog} we finally infer that
$$
\int_0^{T_k(E)/2}\sqrt{A_{l,k}(t)}\, dt+\frac12 \log\left(\frac{\max_{t}\ A_{l,k}(t)}{\min_{t}\ A_{l,k}(t)}\right)$$
$$\le(m+1)\pi+\left(\frac{\pi\, d_{l,k}}{2(m+1)}+\frac{d_{l,k}}{(m+1)^2}-\frac{3\pi(m+1)b_k}{4}\right)\frac{E}{\rho_k^2}+o(E)\quad\mbox{as }E\to0\, .
$$
The statement then follows from assumption \eq{strange3}.\end{proof}

Next, we remark that the function
$$E\mapsto \int_0^{T_k(E)/2}\sqrt{A_{l,k}(t)}\, dt-\frac12 \log\left(\frac{\max_{t}\ A_{l,k}(t)}{\min_{t}\ A_{l,k}(t)}\right)$$
(with a minus sign before the logarithm!) is continuous and, for $E=0$, it is equal to $\frac{T_k(0)}{2}\sqrt{\delta_l}=(m+1)\pi$ in view of
\eq{map} and \eq{stranger}. Therefore, there exists $E_1(l,k)>0$ such that
$$
\int_0^{T_k(E)/2}\sqrt{A_{l,k}(t)}\, dt-\frac12 \log\left(\frac{\max_{t}\ A_{l,k}(t)}{\min_{t}\ A_{l,k}(t)}\right)\ge m\pi\qquad\forall E\le E_1(l,k)\, .
$$
By putting $E_k^l:=\min\{E_1(l,k),E_2(l,k)\}$ and by combining this result with Lemma \ref{suffcond2} we infer that
$$
m\pi\le\int_0^{T_k(E)/2}\sqrt{A_{l,k}(t)}\, dt-\frac12 \log\left(\frac{\max_{t}\ A_{l,k}(t)}{\min_{t}\ A_{l,k}(t)}\right)
$$
$$
\le\int_0^{T_k(E)/2}\sqrt{A_{l,k}(t)}\, dt+\frac12 \log\left(\frac{\max_{t}\ A_{l,k}(t)}{\min_{t}\ A_{l,k}(t)}\right)\le (m+1)\pi\qquad\forall E\le E_k^l\, .
$$
Then a stability result by Burdina \cite{burdina} (see also \cite[Test 3, p.703]{yakubovich}) allows us to conclude that the trivial solution of
\eq{Hill} is stable. By Definition \ref{newdeff} this means that the $k$-th longitudinal mode $\overline \varphi_k$ at energy $E\le E_k^l$
is stable with respect to the $l$-th torsional mode. This completes the proof of Theorem \ref{stable2}.

\begin{remark}\label{several} {\em Instead of the Burdina stability criterion, one may try to use different criteria which may need assumptions other
than \eq{strange3}. For instance, the Zhukovskii criterion (already used in Lemma \ref{suffcond}) needs the assumption that $2d_{l,k}\le(m+1)^2b_k$; since
it is more restrictive than \eq{strange3}, it seems that the Burdina criterion performs better in this situation. But there exist many other criteria, see
\cite{yakubovich}, and some of them could allow to relax further \eq{strange3}. And, perhaps, some criterion would allow to drop any assumption of this kind
after estimating directly $d_{l,k}$ and $b_k$.}\end{remark}

\section{Conclusions: our answers to questions (Q1)-(Q2)-(Q3)}\label{conclusions}

In Section \ref{secsuffcond} we have obtained stability results
for a finite dimensional approximation of \eq{wave-3}. We
analytically proved that if a longitudinal mode is oscillating
with sufficiently small amplitude then it is stable, that is, it
does not transfer energy to torsional modes. The numerical results
in Section \ref{numres} show that if a longitudinal mode is
oscillating with sufficiently large amplitude then it is unstable,
that is, it transfers energy to a torsional mode. These results
are numerically validated in Section \ref{truncsys} thanks to a
more precise approximation of \eq{wave-3}, see
\eq{eq:sistema-completo}; the small discrepancies between the two
approaches are justified in detail. For more numerical simulations
corresponding to other values of $k$ and $l$ we refer to
\cite{bfg}. Overall, recalling Definition \ref{newdeff}, these
results enable us to give the following answer to question {\bf
(Q1)}.\par\medskip
\begin{minipage}{168mm}
{\em Longitudinal oscillations suddenly transform into torsional oscillations because when the flutter energy is reached the longitudinal mode
becomes unstable with respect to a torsional mode.}
\end{minipage}
\par\bigskip

A few days prior to the TNB collapse, the project engineer L.R.\ Durkee wrote a letter (see \cite[p.28]{ammann}) describing the
oscillations which were so far observed at the TNB. He wrote: {\em Altogether, seven different motions have been definitely identified on the main
span of the bridge, and likewise duplicated on the model. These different wave actions consist of motions from the simplest, that of no nodes, to the
most complex, that of seven modes.} On the other hand, we have repeatedly recalled that the day of the collapse {\em the motions, which a moment before
had involved a number of waves (nine or ten) had shifted almost instantly to two}.

In Section \ref{modes} we found the explicit form of both the
longitudinal and torsional modes. We also obtained accurate
approximations of the corresponding eigenvalues when the TNB
parameters are considered. We have analyzed the first and second
torsional modes although we have explained in Figure
\ref{duepponti} why the cables inhibit the appearance of the first
mode. This is confirmed by recent studies in \cite{bergotcivati}.
In Remarks \ref{primarem} and \ref{secondarem} we emphasized that
$A_1(k)>A_2(k)$ provided that $k=8,9,10$ while for lower $k$ we
have that $A_1(k)<A_2(k)$. According to the above reported letter
of Durkee, the TNB never oscillated with $k=8,9,10$ before the day
of the collapse.

The results in \cite{bgz} tell us that the energy transfer occurs when the ratio between the torsional and longitudinal frequencies is small
(close to $1$). And in Remarks \ref{primarem} and \ref{secondarem} we saw that the energy of the $k$-th longitudinal mode for $k=8,9,10$ transfers
earlier to the second torsional mode rather than to the first.

Overall, these results allow us to give the following answer to question {\bf (Q2)}:
\par\medskip
\begin{minipage}{168mm}
{\em Torsional oscillations appear with a node at midspan because the cables inhibit the appearance of the first mode and the longitudinal modes
prior to the switch to torsional modes have an instability threshold with respect to the second torsional mode smaller than with respect to the first
torsional mode.}
\end{minipage}
\par\bigskip

Invoking again the results in \cite{bgz} one reaches the
conclusion that the critical amplitude of oscillation of a
longitudinal mode depends on the ratio between the torsional and
longitudinal frequencies. Table \ref{tableigen} shows that this
ratio reaches its minimum for the 10th longitudinal mode. And
precisely the 10th mode was observed the day of the collapse prior
to the appearance of torsional oscillations. Therefore our answer
to question {\bf (Q3)} is as follows.\par\medskip
\begin{minipage}{168mm}
{\em There are longitudinal oscillations which are more prone to
generate torsional oscillations; in the case of the TNB the most
prone was the 10th longitudinal mode.}
\end{minipage}
\par\bigskip\noindent
\textbf{Acknowledgments.} The first author is partially supported
by the Research Project FIR (Futuro in Ricerca) 2013
\emph{Geometrical and qualitative aspects of PDE's}. The second
and third authors are partially supported by the PRIN project {\em
Equazioni alle derivate parziali di tipo ellittico e parabolico:
aspetti geometrici, disuguaglianze collegate, e applicazioni}. The
three authors are members of the Gruppo Nazionale per l'Analisi
Matematica, la Probabilit\`a e le loro Applicazioni (GNAMPA) of
the Istituto Nazionale di Alta Matematica (INdAM). The second
author was partially supported by the GNAMPA project 2014 {\em
Stabilit\`a spettrale e analisi asintotica per problemi
singolarmente perturbati}. The second author is grateful to G.
Arioli and C. Chinosi for useful discussions during the
preparation of the numerical experiments.

\end{document}